\documentclass[11pt,a4paper,reqno]{amsart}
\usepackage{amsmath,amsfonts,amssymb,amsthm,amscd,color}
\usepackage[english]{babel}
\usepackage[mathscr]{eucal}
\usepackage{graphicx}
\usepackage{hyperref}
\usepackage{dsfont}

\renewcommand{\a}{\alpha}
\renewcommand{\b}{\beta}
\newcommand{\g}{\gamma}

\renewcommand{\d}{\delta}
\newcommand{\D}{\Delta}

\renewcommand{\l}{\lambda}

\newcommand{\na}{\nabla}
\newcommand{\om}{\omega}
\newcommand{\OM}{\Omega}
\renewcommand{\r}{\rho}
\renewcommand{\t}{\tau}

\newcommand{\e}{\varepsilon}
\newcommand{\f}{\varphi}
\newcommand{\F}{\Phi}

\newcommand{\p}{\psi}
\renewcommand{\P}{\Psi}

\newcommand{\N}{{\mathbb N}}
\newcommand{\R}{{\mathbb R}}

\newcommand{\real}{\mathbb R}

\newcommand{\M}{\mathcal{BM}}
\newcommand{\mc}[1]{\mathcal{#1}}
\newcommand{\mbf}[1]{\mathbf{#1}}

\newcommand{\curl}{{\rm curl}\,}

\renewcommand{\div}{{\rm div}\,}

\newcommand{\pd}{\partial}

\newcommand{\supp}{\operatorname{supp\,}}

\newcommand{\loc}{\operatorname{{loc}}}

\newtheorem{theorem}{Theorem}[section]
\newtheorem{proposition}[theorem]{Proposition}
\newtheorem{lemma}[theorem]{Lemma}

\newtheorem{definition}[theorem]{Definition}

\theoremstyle{remark}
\newtheorem{remark}[theorem]{Remark}

\numberwithin{equation}{section}

\begin{document}


\title[Ideal flow around small disks]{Asymptotic behavior of 2D incompressible ideal flow around small disks}
\author[Lacave, Lopes Filho $\&$ Nussenzveig Lopes]{C. Lacave, M. C. Lopes Filho and H. J. Nussenzveig Lopes}

\address[C. Lacave]{Univ Paris Diderot, Sorbonne Paris Cit\'e, Institut de Math\'ematiques de Jussieu-Paris Rive Gauche,
UMR 7586, CNRS, Sorbonne Universit\'es, UPMC Univ Paris 06, F-75013, Paris, France.
$\&$
Univ. Grenoble Alpes, IF, F-38000 Grenoble, France\\
CNRS, IF, F-38000 Grenoble, France.}
\email{christophe.lacave@imj-prg.fr}

\address[M. C. Lopes Filho]{Instituto de Matem\'atica\\Universidade Federal do Rio de Janeiro\\Cidade Universit\'aria -- Ilha do Fund\~ao\\Caixa Postal 68530\\21941-909 Rio de Janeiro, RJ -- BRAZIL.}
\email{mlopes@im.ufrj.br}

\address[H. J. Nussenzveig Lopes]{Instituto de Matem\'atica\\Universidade Federal do Rio de Janeiro\\Cidade Universit\'aria -- Ilha do Fund\~ao\\Caixa Postal 68530\\21941-909 Rio de Janeiro, RJ -- BRAZIL.}
\email{hlopes@im.ufrj.br}

\date{\today}

\begin{abstract}
In this article, we study the homogenization limit of a family of solutions to the incompressible 2D Euler equations in the exterior of a family of $n_k$ disjoint disks with centers $\{z^k_i\}$ and radii $\varepsilon_k$. We assume that the initial velocities $u_0^k$ are smooth, divergence-free, tangent to the boundary and that they vanish at infinity. We allow, but we do not require, $n_k \to \infty$, and we assume $\varepsilon_k \to 0$ as $k\to \infty$.

Let $\gamma^k_i$ be the circulation of $u_0^k$ around the circle $\{|x-z^k_i|=\varepsilon_k\}$. We prove that the homogenization limit retains information on the circulations as a time-independent coefficient. More precisely, we assume that: (1) $\omega_0^k = \mbox{ curl }u_0^k$ has a uniform compact support and converges weakly in $L^{p_0}$, for some $p_0>2$, to $\omega_0 \in L^{p_0}_{c}(\real^2)$, (2) $\sum_{i=1}^{n_k} \gamma^k_i \delta_{z^k_i} \rightharpoonup \mu$ weak-$\ast$ in $\mathcal{BM}(\real^2)$ for some bounded Radon measure $\mu$, and (3) the radii $\varepsilon_k$ are sufficiently small. Then the corresponding solutions $u^k$ converge strongly to a weak solution $u$ of a modified Euler system in the full plane. This modified Euler system is given, in vorticity formulation, by an active scalar transport equation for the quantity $\omega=\curl u$, with initial data $\omega_0$, where the transporting velocity field is generated from $\omega$ so that its curl is $\omega + \mu$. As a byproduct, we obtain a new existence result for this modified Euler system.
\end{abstract}

\maketitle

\tableofcontents

\section{Introduction}

In this article we are concerned with the asymptotic behavior of solutions of the two-dimensional incompressible Euler equations in the exterior of a finite number of vanishingly small disks while, possibly, the number of disks simultaneously tends to infinity. For example, let us say we were interested in approximating fluid flow in the exterior of a rigid wall by a flow in the exterior of a row of small disks following the shape of the wall. This does not work very well. Fluid flow in the exterior of a rigid wall was studied by one of the authors in \cite{lac_euler}; it can be described as the flow due to a vortex sheet whose position is stuck at the wall, but with a vortex sheet strength which is time-dependent. On the other hand, as we will see, flow in the exterior of a finite number of obstacles preserves circulation around each obstacle, according to Kelvin's circulation theorem. Under appropriate hypothesis, as the size of each obstacle vanishes and the number of obstacles increases to approximate the wall, the flows converge to a vortex sheet flow whose position is stuck at the wall but with a time-independent sheet strength, which is, therefore given by the initial data. In other words, conservation of circulation around the small obstacles implies that the homogenization limit does not yield solutions of the Euler equations in the exterior of the wall, but, instead, produces solutions of a suitably modified Euler system, in the full plane. In its simplest form this modified Euler system, in vorticity formulation, is given by a transport equation for the vorticity $\omega$ by a velocity whose curl is the sum of $\omega$ and a Dirac delta supported on the wall. In fact, we will prove that, given any bounded, compactly supported, time-independent, bounded Radon measure $\mu$, the vorticity flow whose velocity is modified by $\mu$ can be approximated by an exterior domain flow, exterior to a finite number of small disks.

More precisely, for each $k = 1,2,\ldots$ we consider a family of $n_k$ disjoint disks with centers $z^k_i$ and radii $\varepsilon_k$, $i=1,\ldots,n_k$.
We denote the disks by $B(z_i^k, \e_k) := \{|x - z^k_i| < \varepsilon_k\}$ and the fluid domain by:
\[
\Omega^k := \real^2 \setminus\Big( \bigcup_{i=1}^{n_k} \overline{ B(z_i^k, \e_k)}\Big).
\]
Let $v$ be smooth, divergence-free vector field in $\Omega^k$, tangent to the boundary, with vorticity,
$w := \mbox{ curl } v$, having bounded support. In non-simply connected domains there are infinitely many vector fields whose curl is $w$. To recover the velocity $v$ from its vorticity $w$ it is necessary to introduce the circulation around each disk:
\begin{equation} \label{circulationdef}
 \g_i^k[v]:= \oint_{\partial B(z_i^k, \e_k)} v\cdot \hat\t ds.
\end{equation}
The above integral is considered in the counter-clockwise sense, hence $\hat\tau =-\hat n^\perp$, where $\hat n$ denotes the outward unit normal vector at $\partial\Omega^k$. Vorticity together with the circulations uniquely determine the velocity, see, for example, \cite{ift_lop_sueur} for a full discussion.

Given $\om_{0}^k\in C^\infty_c(\overline{\Omega^k})$ and $\{\g_i^k\}_{i=1}^{n_k}\subset \R$, the corresponding initial velocity $u_{0}^k$ is the unique smooth solution of
\begin{equation}\label{ini elliptic}
\left\lbrace\begin{aligned}
&\div u_{0}^k=0 &\text{ in }\Omega^k\\
& \curl u_{0}^k=\omega_{0}^k &\text{ in } \Omega^k\\
&u_{0}^k\cdot \hat n =0 &\text{ on }\partial \Omega^k\\
&\g^k_{i}[u_{0}^k]=\g_i^k &\text{ for }i=1,\dots,n_{k}\\
&\lim |u_0^k(x)|=0 &\text{ as }x\to \infty.
\end{aligned}\right.
\end{equation}

 The IBVP for the Euler equations in $\Omega^k$, in velocity formulation, takes the form:
\begin{equation}\label{euler-velocity}
\left\lbrace\begin{aligned}
&\pd_t u^k+u^k \cdot \na u^k=-\na p^k &\text{ in }(0,\infty) \times {\Omega^k}\\
&\div u^k =0 &\text{ in } [0,\infty) \times {\Omega^k} \\
&u^k\cdot \hat{n} =0 &\text{ in }[0,\infty) \times \partial \Omega^k\\
&\lim_{|x|\to\infty}|u^k|=0 & \text{ for }t\in[0,\infty)\\
&u^k(0,x)=u_0^k(x) &\text{ in } \Omega^k,
\end{aligned}\right.
\end{equation}
 where $p^k=p^k(t,x)$ is the pressure.
Existence and uniqueness of a smooth solution $u^k = u^k(t,x), p^k = p^k(t,x)$ of system \eqref{euler-velocity} is due to
K. Kikuchi in \cite{kiku}. We note that, by Kelvin's circulation theorem, the circulations $\gamma_{i}^k[u^k(t,\cdot)]$, $i=1,\dots,n_{k}$, as defined through \eqref{circulationdef}, are conserved quantities.

Taking the curl of the momentum equation in \eqref{euler-velocity} yields a transport equation for the corresponding vorticity $\om^k:=\curl u^k=\pd_1 u_2^k - \pd_2 u_1^k$. This transport equation, together with an elliptic system relating vorticity to velocity and (initial) circulations, comprises the vorticity formulation of the 2D Euler equations:

\begin{equation}\label{vort eq}
\left\lbrace\begin{aligned}
&\pd_t \om^k+u^k \cdot \na \om^k=0 &\text{ in }(0,\infty) \times {\Omega^k}\\
&\div u^k =0 &\text{ in } [0,\infty) \times {\Omega^k} \\
& \curl u^k=\omega^k &\text{ in } [0,\infty) \times\Omega^k\\
&u^k\cdot \hat n =0 &\text{ on }[0,\infty) \times \partial \Omega^k \\
&\g^k_{i}[u^k(t,\cdot)]=\g_i^k &\text{ for }i=1,\dots,n_{k}, \text{ for }t\in[0,\infty)\\
&\lim |u^k(t,x)|=0 &\text{ as }x\to \infty, \text{ for }t\in[0,\infty)\\
&\om^k(0,\cdot)=\om_0^k &\text{ in } \Omega^k.
\end{aligned}\right.
\end{equation}

We emphasize that $\g_i^k$ and $\omega_{0}^k$ are initial data.
We also note, in passing, that, as the transporting velocity is divergence-free, the $L^p$ norm of vorticity is a conserved quantity, for any $p\geq 1$.

Throughout we will adopt the convention that, whenever it is needed to extend $u^k(t,\cdot)$ or $\om^k(t,\cdot)$ to all of $\real^2$, they will be set to vanish in the interior of the disks $B(z_i^k, \e_k)$.

The main purpose of this article is to study the asymptotic behavior of the sequence $\{u^k\}$ when the radii of the disks tend to zero,
both when the number of disks $n_k$ is finite and independent of $k$ and when $n_k \to \infty$ as $k\to \infty$.

It is convenient to fix a scale with respect to which we consider this asymptotic limit. To this end we fix an initial compactly supported vorticity, or eddy, $\omega_0$.

We denote the space of all bounded Radon real measures on $\R^2$ by $\mathcal{BM}(\R^2)$. If $\mu\in \mathcal{BM}(\R^2)$, the total variation of $\mu$ is defined by
\[
\| \mu \|_{\mathcal{BM}}=\sup \Big\{ \int_{\R^2} \varphi \, d\mu \ \vert\ \varphi\in C_{0}(\R^2),\ \| \varphi\|_{L^\infty}\leq 1 \Big\},
\]
where $C_{0}(\R^2)$ is the set of all real-valued continuous functions on $\R^2$ vanishing at infinity. We recall that $\mathcal{BM}(\R^2)$ equipped with the total variation norm is a Banach space. We say that a sequence $(\mu_{n})$ in $\mathcal{BM}(\R^2)$ converges weak-$*$ to $\mu$ if􏰉 $\int_{\R^2}\varphi \,d\mu_{n} \to \int_{\R^2} \varphi\, d\mu$ as $n\to \infty$ for all $\varphi\in C_{0}(\R^2)$.

We assume that the sequence of measures $\sum_{i=1}^{n_k} \gamma_i^k \delta_{z_i^k}$ converges, weak-$\ast$ in the sense of measures, to a compactly supported bounded Radon measure $\mu$. This measure $\mu$ represents the homogenized limit of the circulations $\gamma_i^k$. We assume, without loss of generality, that, for every $k \in \mathbb{N}$, $\{z_i^k\}_{i=1}^{n_k}$ is a set of distinct points.

Our main theorem reads:
\begin{theorem} \label{main}
Fix $\om_0\in L^{p_0}_{c}(\R^2)$ for some $p_0 \in (2,\infty]$. Set $R_0>0$ so that $\supp \omega_{0}\subset B(0,R_{0})$. Let $\{\om_{0}^k\} \subset C^\infty_c(B(0,R_{0}))$ be such that
\[\om_{0}^k \rightharpoonup \om_{0} \text{ weakly in } L^{p_{0}}(\R^2).\]
Consider $\{ \g_i^{k}\}_{i=1\dots n_k}\subset \R$, $\{ z_i^{k}\}_{i=1\dots n_k}\subset B(0,R_{0})$ and suppose that
\[ \sum_{i=1}^{n_k} \g_i^{k} \d_{z_i^k} \rightharpoonup \mu \text{ weak-$*$ in }\ \M(\R^2), \text{ as } k \to \infty,\]
for some $\mu\in \M_c(\R^2)$.

Then there exists a subsequence, still denoted by $k$, a sequence $ \{\varepsilon_k\}\subset \R^*_+$, with $\varepsilon_k \to 0$, together with a vector field $u\in L^p_{\loc}(\R_+\times\R^2)$, for any $p\in [1,2)$, and a function $\om \in L^\infty(\R_+; L^1\cap L^{p_0}(\R^2))$,
such that the solutions $(u^{k} ,\,\om^{k})$ of the Euler equations in
\[\Omega^{k}= \R^2\setminus \Bigl(\bigcup_{i=1}^{n_k} \overline{B(z_i^k,\varepsilon_k)}\Bigl),\]
with initial vorticity $\om_{0}^k\vert_{\Omega^k}$ and initial circulations $(\g_i^{k})$ around the disks, verify
\begin{itemize}
\item[(a)] $ u^{k} \to u$ strongly in $L^p_{\loc}(\R_+\times\R^2)$ for any $p\in [1,2)$;
\item[(b)] $ {\om}^{k} \rightharpoonup {\om}$ weak $*$ in $L^\infty(\R_+; L^{q}(\R^2))$ for any $q\in [1,p_{0}]$;
\item[(c)] the limit pair $(u,{\om})$ verifies, in the sense of distributions:
\begin{equation*}
\left\lbrace\begin{aligned}
&\pd_t {\om}+u\cdot \na {\om}=0 &\text{ in } (0,\infty) \times \R^2 \\
&\div u =0 &\text{ in } [0,\infty) \times \R^2 \\
&\curl u ={\om} + \mu &\text{ in } [0,\infty) \times \R^2 \\
&\lim_{|x|\to\infty}|u|=0 & \text{ for }t\in[0,\infty)\\
&{\om}(0,x)={\om}_0(x) &\text{ in } \R^2.
\end{aligned}\right.
\end{equation*}
\end{itemize}
\end{theorem}

In view of this theorem, the limit behavior of $(u^k,\,\om^k)$ is described by solutions of a modification of the incompressible Euler equations: namely, the vorticity $\omega$ is transported by a divergence-free vector field $u$ which satisfies $\curl u = \om + \mu$. The additional term $\mu$ is reminiscent of the circulation around the evanescent obstacles. We emphasize that the measure $\mu$ above does not depend on time and, also, it is not necessary single-signed, differently from the work of Delort \cite{Delort}.

As a byproduct, Theorem~\ref{main} gives the existence of a global weak solution to this modified Euler problem, where the velocity is obtained by the standard Biot-Savart law of the plane applied to $\omega + \mu$, for $\mu$ a fixed, compactly supported, bounded Radon measure. Therefore, our result may be understood as an extension of the work of Marchioro \cite{marchioro}, in which the author obtained existence of global weak solution when the velocity is the standard Biot-Savart law of the plane applied to $\omega + \alpha \delta_{0}$. Indeed, given any $\mu\in\M_c(\R^2)$, it is easy to find an approximation $\sum_{i=1}^{n_k} \g_i^{k} \d_{z_i^k}$, for example by discretization on a grid.

 Uniqueness for weak solutions of this modified Euler system is an interesting open problem. In the special case where $\mu$ is a finite sum of Dirac masses, $\sum_{i=1}^{n} \gamma_{i} \delta_{z_{i}}$, and when $\omega_{0}\in L^\infty_{c}(\R^2)$, uniqueness was proved by Lacave and Miot in \cite{lac_miot} under the assumption that $\om_{0}$ is initially constant around each of the vortex points $z_{i}$. A key point of their argument was to prove that the non-constant part of $\om$ never meets the trajectories of the vortex points. Uniqueness for general bounded vorticity $\om_{0}$, not necessarily constant around $z_{i}$, is already a challenging open question. For more general $\mu$, it is not clear that a vorticity which initially vanishes around the support of $\mu$ does not intersect this support in finite time. However, we have local uniqueness, up to the time of collision (see Section~\ref{sect final}).

The asymptotic behavior of an inviscid fluid around obstacles shrinking to points was first studied by Iftimie, Lopes Filho and Nussenzveig Lopes in \cite{ift_lop_euler}, where the authors considered only one obstacle which shrinks to one point. Their result can be viewed as a special case of Theorem~\ref{main}: $n_k=1$, $z_1^k=0$, $\mu = \gamma\delta_{0}$, $\omega^k_0=\omega_0 \in C^\infty_c (\R^2\setminus \{0\})$. It was crucial, in \cite{ift_lop_euler}, that there be only one obstacle, because the authors used the explicit form of the Biot-Savart law (giving the velocity in terms of the vorticity) in the exterior of a simply connected compact set. Later, Lopes Filho \cite{milton} treated the case of several obstacles where one of them shrinks to a point. However, the fluid domain was assumed to be bounded: the use of the explicit Biot-Savart law was replaced by standard elliptic ideas in bounded domains. The initial motivation
for the present work was to understand how, and whether it was possible, to approximate ideal flow around a curve by flow outside a finite number of islands approaching the curve. This would be the special case in which $\mu$ is a Dirac supported on a curve. As it turned out, the result we obtained is much more general than we originally sought and, in particular, our original problem is not possible, as we have explained above (see Sections~\ref{sect:curve} and \ref{sect:rate} for more discussions).

From a technical point of view, our main difficulties are how to treat several obstacles without an explicit formula for the Biot-Savart law, together with issues related to unbounded domains (for example: integrability at infinity and analysis of boundary terms which arise when integrating by parts). In particular, we develop in Subsection~\ref{sect : B} (Step 2) an argument based on the inversion map $\mathbf{i}(x):=x/|x|^2$, which sends exterior domains to bounded sets, allowing us to use elliptic tools, such as the maximum principle. For simplicity, we have decided to present all the details in the case where the obstacles are disks, but it is possible to substitute the disks for a more general shape. Let $\mathcal{K}$ be a smooth, simply connected compact set containing zero and consider obstacles of the form $z_i^k+ \varepsilon_k \mathcal{K}$. Then, in \cite{BLM}, these results have been extended to these kinds of obstacles and it was shown that the limit does not depend on the shape of $\mathcal{K}$.

The remainder of this article is organized in four sections. In the next section, we recall precisely how to recover velocity from vorticity and circulation around boundary components and we introduce basic notation, to be used throughout the paper.
Section~\ref{sect u0} is dedicated to deriving uniform estimates on velocity in $L^p$, for all $p\in [1,2)$. The method we use is to construct an explicit correction of the Biot-Savart law from the formulas in the exterior of a single disk, and then to compare this with $u^k$ and with the standard Biot-Savard law in the full plane, applied to $\omega+\mu$. Comparing this correction with the Biot-Savart formula in the full plane (Step 1) will not be too hard, because we have explicit expressions. However, the comparison with $u^k$ (Step 2) will be more complicated, and will justify the use of the inversion mapping. The additional difficulty comes from the fact that we are considering non-zero circulations. Indeed, when an obstacle shrinks to a point, a Dirac mass in the curl of $u$ appears as an asymptotic limit, and the velocity associated to $\gamma \delta_z$ is of the form $\gamma \frac{(x-z)^\perp}{2\pi|x-z|^2}$, which only belongs to $L^p_{\loc}(\R^2)$ for $p<2$, but not to $L^2_{\loc}(\R^2)$. The stability of the Euler equations under Hausdorff approximation of the fluid domain is a recent result of G\'erard-Varet and Lacave \cite{GV-Lac,GV-Lac2}, but the authors considered obstacles with $H^1$ positive capacity, in order to use $L^2$ arguments (the Sobolev capacity of a material point is zero, see \cite{GV-Lac,GV-Lac2} for details). This difficulty also explains why the authors in \cite{BLM,LM} treat the case of zero circulations (see Section~\ref{sect:rate}).

In Section~\ref{sect time}, we use these $L^p$ estimates to prove our main theorem.

Finally, we collect in the last section some final comments and remarks. In particular, we prove a couple of uniqueness results for the limit problem, as mentioned above (the case where $\mu$ is a finite number of Dirac masses, and local uniqueness if the initial vorticity is supported far from $\supp \mu$). We will also discuss the case where the balls are uniformly distributed on a curve. In particular, we make rigorous the observation that the asymptotic behavior of solutions, as $n_k\to \infty$, is not a solution of the Euler equations around a curve, as described in \cite{lac_euler}. This disparity can be attributed to the fact that there is no control on the disk radii $\varepsilon_k$ in Theorem~\ref{main}. Recently, such control was investigated in \cite{BLM,LM}, in the particular case of zero circulations.

We conclude this introduction by recalling that the study of the flow through a porous medium has a long history in the homogenization framework. We refer to Cioranescu and Murat \cite{CM82} for the Laplace problem, and to Tartar \cite{Tartar80} and Allaire \cite{Allaire90a,Allaire90b} for the Stokes and Navier-Stokes equations. There are also many works concerning viscous flow through a sieve, see, e.g., \cite{conca1,conca2,Diaz99,SP1} and references therein. For a modified Euler equations (weakly non-linear), Mikeli\'c-Paoli \cite{MikelicPaoli} and Lions-Masmoudi \cite{LionsMasmoudi} also obtained a homogenized limit problem.

\medskip

{\it Notation:} we adopt the convention $(a,b)^\perp =(-b,a)$ and $\nabla^\perp f = (\nabla f)^\perp$. In the sequel, C will denote a constant independent of the underlying parameter, the value of which can possibly change from a line to another.

\section{Preliminaries about the Biot-Savart law}\label{sect : biot}

 The purpose of this section is to introduce basic notation and recall facts about the Biot-Savart law, which expresses the velocity in terms of the vorticity and circulations. The details can be found, for example, in \cite{kiku,milton,ift_lop_sueur}.

Let $\Omega$ be the exterior of $n$ disjoint disks:
\[ \Omega= \R^2 \setminus \Bigl( \bigcup_{i=1}^n \overline{ B(z_i, r_i)}\Bigl).\]

Consider also a disk with $n$ (circular) holes:
\[ \widetilde\Omega= B(z_0,r_0) \setminus \Bigl( \bigcup_{i=1}^n \overline{B(z_i,r_i)}\Bigl)\]
where $\bigcup_{i=1}^n \overline{B(z_i,r_i)}\subset B(z_0,r_0)$.

Let $U \subseteq \real^2$ be a smooth domain, which can be either $\Omega$ or $\widetilde\OM$. Given a smooth function $f$, compactly supported in $U$, and $(\g_i)\in \R^n$, we consider the following elliptic system:

\begin{equation}\label{elliptic}
\left\lbrace\begin{aligned}
&\div u=0 &\text{ in } U \\
&\curl u =f &\text{ in } U \\
&u\cdot \hat{n} =0 &\text{ on } \partial U \\
&\g_i[u]=\g_i& \text{ for }i=1,\dots,n ,
\end{aligned}\right.
\end{equation}
where $\g_i[u]$ was defined in \eqref{circulationdef}.

When $U= \Omega$ we will assume an additional condition at infinity, namely, $\lim_{|x|\to\infty}|u|=0 $.

\subsection{Harmonic part}\

We introduce families of harmonic functions called harmonic measures.
\begin{itemize}
 \item in the unbounded domains, the harmonic measures are the functions $\{\F_i\}_{i=1\dots n}$, the unique solutions of
\begin{equation*}
\left\lbrace\begin{aligned}
&\D \F_i= 0 &\text{ in }{\Omega} \\
&\F_i = \d_{i,j} &\text{ on } \pd B(z_j,r_j)\\
&\F_i \text{ has a finite limit at }\infty,&
\end{aligned}\right.
\end{equation*}(where $\d_{i,j}$ is the Kronecker delta) ;
\item in bounded domains, we introduce the functions $\{\tilde\F_i\}_{i=0\dots n}$ as the unique solutions of
\begin{equation}\label{tildePhi}
\left\lbrace\begin{aligned}
&\D \tilde\F_i= 0 &\text{ in }{\widetilde\Omega} \\
&\tilde\F_i = \d_{i,j} &\text{ on } \pd B(z_j,r_j).
\end{aligned}\right.
\end{equation}
\end{itemize}

By uniqueness of the 	Dirichlet problem for the Laplacian, we have:	
$$
\sum_{i=1}^n \F_i \equiv 1 \text{ in }\Omega, \qquad \sum_{i=0}^n \tilde\F_i \equiv 1 \text{ in }\widetilde\Omega,
$$
and for any $i$
$$ \F_i (x)\in[0,1]\ \forall x\in \Omega , \qquad \tilde\F_i (x)\in[0,1]\ \forall x\in \widetilde\Omega.
$$
Indeed,
\begin{itemize}
 \item the bounded case follows directly from the maximum principle ;
 \item in the unbounded case, let us assume that one of the disks is centered at the origin, let say $z_1=0$. Next, we use the inversion $\mathbf{i}(x):= \frac{x}{|x|^2}$ (see Lemma~\ref{disk_inversion} for its properties), which maps $\Omega$ to a bounded domain (included in $B(0,1/r_1)$) with $n-1$ obstacles. Next, we introduce $\tilde \F_{i}(x):= \F_{i}\circ \mbf{i}^{-1}(x)$ which verifies \eqref{tildePhi} in $\mbf{i}(\Omega)$ (see Lemma~\ref{disk_inversion}). Hence, the result follows from the bounded case.
\end{itemize}

In the bounded case, we will use later that $\{\nabla^\perp \tilde\F_i\}_{i=1\dots n}$ is a basis for the harmonic vector fields (i.e, vector fields which are both solenoidal and irrotational and which are tangent to the boundary). Indeed, let $H$ be a harmonic vector field, then:
\begin{itemize}
\item the divergence-free condition reads as $H(x)=\nabla^\perp \p(x)$ for some stream function $\p$;
\item the tangency condition means $\p$ is constant on each $\pd B(z_i,r_i)$ for $i=0\dots n$. We denote by $c_{i}$ these constants. As $\p$ is defined up to a constant, we can determine uniquely $\p$ by assuming $c_{0}=0$;
\item the curl-free condition implies $\Delta \p=0$ in $\widetilde \Omega$.
\end{itemize}
By uniqueness, it follows that
$$
\p = \sum_{i=1}^n c_{i} \tilde\F_{i}, \quad H = \sum_{i=1}^n c_{i} \nabla^\perp \tilde\F_{i}.
$$

\bigskip

Next, we seek a family of harmonic vector fields such that the circulation around $B(z_j,r_j)$ is $\delta_{i,j}$.

In bounded domains, we introduce the vector field $\{\tilde X_i\}_{i=1\dots n}$ as the unique solution of
\begin{equation*}
\left\lbrace\begin{aligned}
&\div \tilde X_{i} = \curl \tilde X_{i}= 0 &\text{ in }{\widetilde\Omega} \\
& \tilde X_{i}\cdot \hat n =0 &\text{ on } \partial {\widetilde\Omega}\\
&\oint_{\pd B(z_j,r_j)} \tilde X_{i} \cdot \hat \tau \, ds = \delta_{i,j} &\text{ for } j=1\dots n.
\end{aligned}\right.
\end{equation*}
It is easy to see that the family $\{ \tilde X_i\}_{i=1\dots n}$ is also a basis for the harmonic vector fields, since this family is clearly linearly
independent and we already know that the dimension of the space of harmonic vector fields is precisely $n$. Alternatively, we can also express the
$\tilde X_i$'s in terms of stream functions $\tilde X_{i}=\nabla^\perp \tilde\P_{i}$, where $\tilde\Psi_{i}$ is the unique solution of
\begin{equation*}
\left\lbrace\begin{aligned}
&\Delta \tilde\Psi_{i} = 0 &\text{ in }{\widetilde\Omega} \\
& \partial_{\tau} \tilde\Psi_{i} = 0 &\text{ on } \partial {\widetilde\Omega}\\
&\tilde \Psi_{i} = 0 &\text{ on } \partial B(z_0,r_0)\\
&\oint_{\pd B(z_j,r_j)} \partial_{n} \tilde\Psi_{i} \, ds = -\delta_{i,j} &\text{ for } j=1\dots n.
\end{aligned}\right.
\end{equation*}
The condition $\partial_{\tau} \tilde\Psi_{i} = 0$ means that $\tilde\Psi_{i}$ is constant on each $\pd B(z_i,r_i)$ for $i=0\dots n$. Then, we can uniquely determine a change of basis matrix, i.e. constants $c_{i,j}$ such that
$$
\tilde \P_{i} = \sum_{j=1}^n c_{i,j} \tilde\F_{j}, \quad \tilde X_{i} = \sum_{j=1}^n c_{i,j} \nabla^\perp \tilde\F_{j}.
$$

In unbounded domains, we introduce the vector fields $\{X_i\}_{i=1\dots n}$ as the unique solutions of
\begin{equation*}
\left\lbrace\begin{aligned}
&\div X_{i} = \curl X_{i}= 0 &\text{ in }{\Omega} \\
& X_{i}\cdot \hat n =0 &\text{ on } \partial {\Omega}\\
& |X_{i}(x)|\to 0 & \text{ when } |x|\to \infty\\
&\oint_{\pd B(z_j,r_j)} X_{i} \cdot \hat \tau \, ds = \delta_{i,j} &\text{ for } j=1\dots n.
\end{aligned}\right.
\end{equation*}
Clearly, the family $\{ X_i\}_{i=1\dots n}$ is a basis for the harmonic vector fields. In addition, we have the following asymptotic expansion by Laurent series:
\begin{equation}\label{Xi infini}
X_i (x) = \frac{1}{2\pi} \frac{x^\perp}{|x|^2}+ \mc{O}\left(\frac{1}{|x|^2}\right),\text{ as } |x|\to \infty.
\end{equation}

To reformulate this discussion in terms of stream functions, we first observe that we cannot assume that $\Psi_{i}$ goes to zero at infinity. One way to choose $\Psi_{i}$ uniquely is to assume that $c_{0,i}=0$ in the Laurent expansion
$$\Psi_{i}=\frac1{2\pi}\ln|x| + c_{0,i}+\mc{O}\Big(\frac{1}{|x|}\Big).$$
Consequently, we can reformulate our description of the space of harmonic vector fields in terms of the stream functions $X_{i}=\nabla^\perp \P_{i}$,
where, for each $i=1,\ldots,n$, $\Psi_{i}$ is the unique solution of
\begin{equation*}
\left\lbrace\begin{aligned}
&\Delta \Psi_{i} = 0 &\text{ in }{\Omega} \\
& \partial_{\tau} \Psi_{i} = 0 &\text{ on } \partial {\Omega}\\
& |\P_{i}(x) - \frac1{2\pi} \ln |x| |\to 0 & \text{ when } |x|\to \infty\\
&\oint_{\pd B(z_j,r_j)} \partial_{n} \Psi_{i} \, ds = -\delta_{i,j} &\text{ for } j=1\dots n.
\end{aligned}\right.
\end{equation*}
The stream functions $\Psi_{i}$ above satisfy:
\begin{equation}\label{Psi infini}
\P_i(x) = \frac{1}{2\pi} \ln |x|+ \mc{O}\Big(\frac{1}{|x|}\Big),\text{ as } |x|\to \infty.
\end{equation}

We remark that the precise form of the term $\mc{O}(|x|^{-1})$ above depends on the shape of the domain, so the expansions \eqref{Xi infini} and \eqref{Psi infini}, will not be uniform in $k$, when we go back to our original problem. However, these estimates will be only used to justify certain integrations by parts. Moreover, we note that $\Psi_{i}$ and $\F_{i}$ do not have the same behavior at infinity, and that $\{\nabla^\perp \F_i\}_{i=1\dots n}$ is not a basis for the harmonic vector fields in unbounded domains.

\subsection{The Laplacian-inverse}\

In this subsection, we focus on the solution of $\D \P= f$ with Dirichlet boundary condition. We denote by $\P_D = \P_D[f]$ the solution operator
of the Dirichlet Laplacian in $\Omega$. More precisely, $\psi = \P_D[f]$ is the unique solution of
\begin{equation*}
\left\lbrace\begin{aligned}
&\D \psi= f &\text{ in }{\Omega} \\
&\psi = 0 &\text{ on } \pd \Omega\\
&\nabla\psi(x) \to 0 &\text{ when } |x|\to \infty ,
\end{aligned}\right.
\end{equation*}
and we denote by $\tilde\P_D = \tilde\P_D[ f]$ the corresponding solution operator in $\widetilde\Omega$, so
that $\tilde\psi = \tilde\P_D[f]$ is the unique solution of
\begin{equation*}
\left\lbrace\begin{aligned}
&\D \tilde\psi= f &\text{ in }{\widetilde\Omega} \\
&\tilde\psi = 0 &\text{ on } \pd \widetilde\Omega.
\end{aligned}\right.
\end{equation*}

We also define
\[ K[f]:= \na^\perp \P_D[f], \qquad \tilde K[ f]:= \na^\perp \tilde \P_D[ f], \]
which are divergence free, tangent to the boundary and verify
$$
\curl K[f] =f \text{ on }\Omega, \qquad \curl\tilde K[ f] = f \text{ on }\widetilde\Omega.
$$
Moreover, we have for any $i=1\dots n$:
\begin{equation}\label{circu of K}
\oint_{\pd B(z_i,r_i)} K[f] \cdot \hat \tau \, ds = -\int_{\Omega} \F_{i}(y)f(y)\, dy, \qquad \oint_{\pd B(z_i,r_i)} \tilde K[ f] \cdot \hat \tau \, ds = -\int_{\widetilde\Omega}\tilde \F_{i}(y) f(y)\, dy.
\end{equation}

One relevant issue in the remainder of this article is the behavior at infinity of $\P_D[f]$ and of $K[f]$.
We observe that if $\supp f\subset B(0,R)$, then there exists $C_{R}$ such that, for all $|x|\geq 2R$ we have
\begin{equation}\label{psi_prop}
|\P_D[f](x)| \leq C_{R} \|f\|_{L^1}, \qquad |K[f](x)| \leq \frac{C_{R} \|f\|_{L^1}}{|x|^2}.
\end{equation}

In these inequalities, the constant $C_{R}$ depends on the size of the support of $f$ and of the domain $\Omega$. As the support of the vorticity moves, the above estimates, with $f=\omega^k(t,\cdot)$, will depend on the time. But, as remarked in the previous subsection, even if the constant in \eqref{psi_prop} depends on $t$ and $k$, it will be useful to justify certain integrations by parts, at $t$ and $k$ fixed. For a proof of \eqref{psi_prop}, see, for example, \cite{ift_lop_sueur}.

\subsection{Biot-Savart law}\label{sect BS}\

Next we introduce notation for the Biot-Savart law. If $f$ is a smooth function, compactly supported in $\overline{\Omega}$ (or supported in $\widetilde\OM$) and $(\g_i)\in \R^n$, the unique solution $u$ of the elliptic system \eqref{elliptic} is given by:
\begin{eqnarray}
&u(x)=K[f](x)+\sum_{i=1}^{n} \a_i[ f ] X_i(x) &\text{ in } \Omega \label{biot}\\
&u(x)=\tilde K[f](x)+\sum_{i=1}^{n} \tilde\a_i[ f ] \tilde X_i(x) &\text{ in }\widetilde \Omega \nonumber,
\end{eqnarray}
where
\[
\a_i[ f]= \int_{\Omega} \F_i (y) f (y)\, dy +\g_i\quad\text{and}\quad \tilde\a_i[ f]= \int_{\widetilde\Omega} \tilde\F_i (y) f (y)\, dy +\g_i,
\]
see \cite{ift_lop_sueur} and \cite{milton}.

In what follows, we use the notation $m_i[f]:=\displaystyle \int_{\Omega} \F_i(y) f(y) \, dy$. We add a superscript $k$ to all the functions defined above, in order to keep in mind the dependence on the domain. For example, the unique vector field $u_0^k$ verifying \eqref{ini elliptic} is given by:
\begin{equation*}
u^k_0(x)=K^k[\om_0^k\vert_{\Omega^k}](x)+\sum_{i=1}^{n_{k}} \a_i^k[ \om_0^k\vert_{\Omega^k} ] X_i^k(x)\quad \text{with}\quad \a_i^k[ \om_0^k\vert_{\Omega^k} ]= \int_{\Omega^k} \F_i^k (y) \om_{0}^k (y)\, dy +\g_i^k.
\end{equation*}

\bigskip

Our basic strategy is to deduce strong convergence for the velocity from \eqref{biot} and from $L^p$ estimates of the vorticity.

\section{Convergence for fixed time}\label{sect u0}

For a function $v$ defined in the domain
$$
\Omega^{\e,k}:= \R^2\setminus \Bigl(\bigcup_{i=1}^{n_k} \overline{B(z_i^k,\e)}\Bigl),
$$
we denote by $Ev$ its extension of to the plane by setting $Ev$ to vanish outside $\Omega^{\e,k}$. Let $K_{\R^2}$ be the Biot-Savart operator in the full plane. For $f \in C^{\infty}_c(\overline{\Omega^{\e,k}})$ we write:
\[ K_{\R^2}[f]=\frac{1}{2\pi} \frac{x^\perp}{|x|^2}\ast f. \]

We seek to prove that the Biot-Savart formula \eqref{biot} converges to $K_{\R^2}[\mu+ f]$ when $k \to \infty$. It is natural to study the following decomposition:
\begin{equation*}\begin{split}
K_{\R^2}[\mu+ E f\vert_{\Omega^{\e,k}}](x) - E\Bigl(K^{\e,k}[f\vert_{\Omega^{\e,k}}](x)&+\sum_{i=1}^{n_k} \a_i^{\e,k}[f] X_i^{\e,k}(x)\Bigl)\\
=&\Bigl( K_{\R^2}[\mu]-K_{\R^2}\left[\sum_{i=1}^{n_k} \g_i^k \d_{z_i^k}\right]\Bigl)(x) \\
&+\sum_{i=1}^{n_k} \g_i^k \Bigl( \frac{(x-z_i^k)^\perp}{2\pi |x-z_i^k|^2} - EX_i^{\e,k}(x)\Bigl)\\
&+\Bigl( K_{\R^2}[E f\vert_{\Omega^{\e,k}}]-EK^{\e,k}[f\vert_{\Omega^{\e,k}}]-\sum_{i=1}^{n_k} m_i^{\e,k}[f] EX_i^{\e,k}\Bigl)(x)\\
=:& A_k(x) + B_{\e,k}(x) + C_{\e,k}[f](x) .
\end{split}\end{equation*}

The goal of this section is to prove a serie of three propositions associated to this decomposition.

The first proposition is concerned with the first term, where we convert the weak-$*$ convergence in $\M_c(\R^2)$ to strong convergence of the associated velocities.
\begin{proposition} \label{A_n}
Let $R_{0}>0$ fixed. If a sequence $\{ \g_i^{k}\}_{i=1\dots n_k}\subset \R$ and $\{ z_i^{k}\}_{i=1\dots n_k}\subset B(0,R_{0})$ verifies 
\[ \sum_{i=1}^{n_k} \g_i^{k} \d_{z_i^k} \rightharpoonup \mu \text{ weak-$*$ in }\ \M(\R^2) \ \text{ as } k\to \infty\]
for some $\mu\in \M_c(\R^2)$, then for any $p\in (1,2)$ we can decompose $A_{k}$ in two parts $A_{k}=A_{k,1}+A_{k,2}$ such that
\[
\| A_{k,1}\|_{L^p(\R^2)} + \| A_{k,2}\|_{L^{p^*}(\R^2)} \to 0 \text{ as } k\to \infty,
\]
where $p^*=(\frac1p-\frac12)^{-1}\in (2,\infty)$.
\end{proposition}

The result in Proposition~\ref{A_n} fixes $k$ such that $K_{\R^2}[\mu]$ is well approximated. In the following proposition, the $k$ is fixed, and we study the limit $\e\to 0$ for the harmonic vector fields.

\begin{proposition} \label{B_n}
Let $k\in \N$ fixed and $\{ z_i^{k}\}_{i=1\dots n_k}\subset \R^2$ given.
Then, for any $i=1\dots n_k$, there exists $v_{i}^\e$ such that
\[
\Bigl\| \frac{(x-z_i)^\perp}{2\pi |x-z_i|^2} -Ev_{i}^\e(x) \Bigl\|_{L^p(\R^2)}+\|Ev_{i}^\e -EX_i^{\e,k} \Bigl\|_{L^2(\R^2)}\to 0 \text{ as } \e \to 0,
\]
for any $p\in[1,2)$.
\end{proposition}

Next, we are interesting about the convergence of $C_{\e,k}$. We will see later that it is important to establish it uniformly in the $(L^1\cap L^{p_0})$-norm of $f$.
\begin{proposition} \label{C_n} Let $M_0>0$, $k\in \N$, $\{ z_i^{k}\}_{i=1\dots n_k}\subset \R^2$ fixed.
We have that
\[
\Bigl\| K_{\R^2}[f\vert_{\Omega^{\e,k}}]-EK^{\e,k}[f\vert_{\Omega^{\e,k}}]-\sum_{i=1}^{n_k} m_i^{\e,k}[f] EX_i^{\e,k}\Bigl\|_{L^2(\R^2)} \to 0 \text{ as } \e\to 0,
\]
uniformly in $f$ verifying:
\[ \| f \|_{L^1\cap L^{p_0}(\R^2)} \leq M_0 \quad \text{and} \quad f \text{ is compactly supported.}
\]
\end{proposition}

In the last part of this section, we will construct $\varepsilon_k$ such that Theorem~\ref{main} can be proved.

\subsection{The number of obstacles (proof of Proposition~\ref{A_n})}\label{number_obstacle}\

Since $\{ z_i^{k}\}_{i=1\dots n_k}\subset B(0,R_{0})$, we have that $\supp \mu \subset \overline{B(0,R_{0})}$. From the weak-$*$ convergence
\[ \sum_{i=1}^{n_k} \g_i^{k} \d_{z_i^k} \rightharpoonup \mu \text{ weak-$*$ in }\ \M(\R^2) \quad \text{when } k\to \infty,\]
we infer that $\| \sum_{i=1}^{n_k} \g_i^{k} \d_{z_i^k} \|_{\M(\R^2)}$ is uniformly bounded in $k$ (see \cite[Prop. 3.13]{Brezis}, which is a consequence of the Banach-Steinhaus theorem). We have, by duality, that $\M(B(0,R_{0}+1))$ is compactly imbedded in $W^{-1,p}(B(0,R_{0}+1))$ for any $p\in (1,2)$, since $W^{1,p'}_0(B(0,R_0+1))$ is compactly imbedded in $C_{0}(B(0,R_0+1))$, where $p'=p/(p-1)$. Hence we get that
\[ \sum_{i=1}^{n_k} \g_i^{k} \d_{z_i^k} - \mu \to 0 \text{ strongly in }\ W^{-1,p}(B(0,R_{0}+1)) \quad \text{when } k\to \infty,\]
for any $p\in (1,2)$. Let us fix $p\in (1,2)$. We recall that a distribution $h$ in $W^{-1,p}(\OM)$ can be decomposed as $h=f_{0}+ \partial_{1} f_{1}+\partial_{2} f_{2}$, where $f_0$, $f_{1}$ and $f_{2}$ are in $L^p(\OM)$ and $\max_{l=0\dots 2} \| f_{l}\|_{L^{p}(\OM)} \leq \|h\|_{W^{-1,p}(\OM)}$(see, for example, \cite[Prop. 9.20]{Brezis} for this decomposition). Consequently, there exists $f_{0}^k$, $f_{1}^k$ and $f_{2}^k$ belonging in $L^{p}(\R^2)$ such that
\[
\sum_{i=1}^{n_k} \g_i^{k} \d_{z_i^k} - \mu = f_{0}^k+\partial_{1} f_{1}^k + \partial_{2} f_{2}^k \quad \text{ in } \R^2,
\]
with $\max_{l=0,\,1,\, 2} \| f_{l}^k\|_{L^{p}(\R^2)} \to 0$.

Now, we apply the Biot-Savart law in $\R^2$:
\begin{equation}\label{BS BM}
K_{\R^2}[\sum_{i=1}^{n_k} \g_i^k \d_{z_i^k}]-K_{\R^2}[\mu]= K_{\R^2}[f_{0}^k]+ \partial_{1} K_{\R^2}[f_{1}^k] + \partial_{2} K_{\R^2}[f_{2}^k].
\end{equation}
For the first right hand side term, we apply the Hardy-Littlewood-Sobolev Theorem (see e.g. \cite[Theo. V.1]{Stein} with $\alpha=1$) where $\frac1{p}=\frac1{p^*}+\frac12$ to get:
\[
\| K_{\R^2}[f_{0}^k] \|_{L^{p^*}(\R^2)} \leq C_{p} \|f_{0}^k\|_{L^{p}(\R^2)}.
\]
This inequality holds for $p\in (1,2)$, and $p^*$ belongs to $(2,\infty)$.

Concerning the other terms in \eqref{BS BM}, we infer from the Calder\'on-Zygmund inequality that
\begin{eqnarray*}
 \| \partial_{1} K_{\R^2}[f_{1}^k] + \partial_{2} K_{\R^2}[f_{2}^k] \|_{L^p(\R^2)}
&\leq& \| \nabla K_{\R^2}[f_{1}^k] \|_{L^p(\R^2)}+ \| \nabla K_{\R^2}[f_{2}^k] \|_{L^p(\R^2)}\\
&\leq& C_{p} (\|f_{1}^k\|_{L^p(\R^2)}+\|f_{2}^k\|_{L^p(\R^2)}).
\end{eqnarray*}

Setting $A_{k,1}:=\partial_{1} K_{\R^2}[f_{1}^k] + \partial_{2} K_{\R^2}[f_{2}^k]$ and $A_{k,2}:=K_{\R^2}[f_{0}^k]$ ends the proof of Proposition~\ref{A_n}.

\subsection{The harmonic part (proof of Proposition~\ref{B_n})} \label{sect : B}\

In this subsection $k$ is fixed, so to simplify the notations we will omit the parameter $k$ in all the functions and domains. As $n_k$ is fixed, then we study the behavior of the flow around $n(=n_k)$ obstacles $\{B(z_i, \e)\}_{i=1\dots n}$ which shrink to $n$ points as $\e\to 0$. We denote by $\r=\min_{i\neq j} |z_i-z_{j}|$ and $B_i^\e := B(z_i,\e)$ with $\e<\r/2$. Let $i$ be fixed, the goal of this subsection is to compare
\[ \frac{1}{2\pi}\frac{(x-z_i)^\perp}{|x-z_i|^2} \quad \text{and} \quad EX_i^\e .
\]

The first vector field is not tangent to $B_j^\e$ for $j\neq i$, whereas $X_i^\e$ is. We introduce a vector field $v_{i}^\e$ which has an explicit form and verifies some of the properties of $X_i^\e$.
If we denote by $\f:[0,\infty)\to[0,\infty) $ a non-increasing function such that $\f(s)=1$ if $s<\r/4$ and $\f(s)=0$ if $s>\r/2$, we introduce cut-off functions: $\f_j(x):= \f(|x-z_j|)$ which verifies
\begin{equation}\label{cut_off}
\left\lbrace\begin{aligned}
&\f_j(x) \equiv 1 &\text{ in } B(z_j,\r/4) \\
&\f_j(x) \equiv 0 &\text{ in } B(z_j,\r/2)^c.
\end{aligned}\right.
\end{equation}
We use the inversion with respect to $B_j^\e$ to define, for each $i \neq j$,
\begin{equation}\label{z ast}
z_i^{\ast j,\e}:= z_j+\e^2\frac{z_i-z_j}{|z_i-z_j|^2}\in B(z_{j},\varepsilon),
\end{equation}
which allows us to introduce the vector fields $v_i^\e$ by:
\begin{equation}\label{v}
v_i^\e(x):= \frac{1}{2\pi}\frac{(x-z_i)^\perp}{|x-z_i|^2}-\na^\perp \Bigl(\sum_{j\neq i} \frac{\f_j(x)}{2\pi}\ln \frac{|x-z_i^{\ast j,\e}|}{|x-z_j|} \Bigl).
\end{equation}
The properties of such a vector field are listed here:
\begin{lemma}\label{lem correction}
We have that:
\begin{enumerate}
\item $v_{i}^\e$ is divergence free ;
\item $v_{i}^\e$ is tangent to the boundary $\partial \Omega^\varepsilon$ ;
\item the circulations of $v_{i}^\e$ around $B_j^\e$ are equal to $\delta_{i,j}$, for all $j=1\dots n$ ;
\item $\displaystyle \curl v_{i}^\e(x)= - \frac1{2\pi}\sum_{j\neq i} \ln \frac{|x-z_i^{\ast j,\e}|}{|x-z_j|} \Delta \f_j(x)- \frac1{\pi}\sum_{j\neq i} \Bigl(\frac{x-z_i^{\ast j,\e}}{|x-z_i^{\ast j,\e}|^2}- \frac{x-z_j}{|x-z_j|^2}\Bigl)\cdot \nabla \f_j(x)$ in $\Omega^\varepsilon$.
\end{enumerate}
\end{lemma}
\begin{proof}
The first point is obvious, because this vector field is a perpendicular gradient. The last point is just a basic computation, noting that $\Delta \frac1{2\pi}\ln |x-P| = \delta_{P}$ (where $\delta_{P}$ denotes the Dirac mass at the point $P$) and that $z_i^{\ast j,\e}\in B(z_{j},\e)$, so $\Delta \ln \frac{|x-z_i^{\ast j,\e}|}{|x-z_j|}$ is equal to zero outside $B(z_{j},\e)$.

Concerning (2) and (3), we note that $v_{i}^\e$ is equal to
$$\frac1{2\pi} \frac{(x-z_i)^\perp}{|x-z_i|^2}$$
in a neighborhood of $B_{i}^\e$ (namely for any $x\in B(x_{i}, \rho/4)\setminus B(x_{i},\e)$), whereas it is equal to
$$
 \frac{1}{2\pi}\Bigl(\frac{x-z_i}{|x-z_i|^2}- \frac{x-z_i^{\ast j,\e}}{|x-z_i^{\ast j,\e}|^2} + \frac{x-z_j}{|x-z_j|^2} \Bigl)^\perp= -\frac{1}{2\pi} \na^\perp \ln \frac{|x-z_i^{\ast j,\e}|}{|x-z_i| |x-z_j|}
$$
in a neighborhood of $B_{j}^\e$ for $j\neq i$ (namely for any $x\in B(x_{j}, \rho/4)\setminus B(x_{j},\e)$). With the first form, it is clear that $v_{i}^\e$ is tangent to $B_{i}^\e$ and that
$$
\oint_{\partial B_{i}^\e} v_{i}^\e \cdot \hat \tau \, ds = 1.
$$
Thanks to the definition of $z_i^{\ast j,\e}$ \eqref{z ast}, we can easily prove that
\[
\frac{|x-z_i^{\ast j,\e}|}{|x-z_i| |x-z_j|} =\frac1{|z_{i}-z_{j}|}, \quad \forall x \in \partial B_{j}^\e,
\]
hence we deduce that the normal component of the perpendicular gradient is equal to zero, i.e. $v_{i}^\e$ is tangent to $B_{j}^\e$. Moreover, there exists $r>0$ such that $B(z_i^{\ast j,\e},r)\Subset B(z_{j},\varepsilon)$, so we get by Stokes formula that:
\begin{align*}
 \oint_{\partial B_{j}^\e} v_{i}^\e \cdot \hat \tau \, ds
&= \frac{1}{2\pi}\int_{B_{j}^\e} \curl \frac{(x-z_i)^\perp}{|x-z_i|^2} \, dx + \frac{1}{2\pi}\oint_{\partial B_{j}^\e} \frac{(x-z_j)^\perp}{|x-z_j|^2}\cdot \hat \tau \, ds + \frac{1}{2\pi}\oint_{\partial B(z_i^{\ast j,\e},r)} \frac{(x-z_i^{\ast j,\e})^\perp}{|x-z_i^{\ast j,\e}|^2}\cdot \hat \tau \, ds\\
&= 0+1-1=0,
\end{align*}
which ends the proof of the lemma.
\end{proof}

\begin{remark} \label{rem biholo}
The fact that the obstacles are disks simplifies the expression of $v_i^\e$. Otherwise, we would have to consider $T_j^\e$ the biholomorphism between $(\mathcal{K}_j^\e-z_j)^c$ and the exterior of the unit disk, and we would define:
\begin{equation*}
v_i^\e(x) := \frac{1}{2\pi}\frac{(x-z_i)^\perp}{|x-z_i|^2}+\na^\perp \Bigl(\sum_{j\neq i} \f_j\frac{1}{2\pi}\ln \frac{|T_j^\e(x-z_j)-T_j^\e(z_i-z_j)| |T_j^\e(x-z_j)|}{\frac{|x-z_{i}|}{\e} |T_j^\e(x-z_j)-\frac{T_j^\e(z_i-z_j)}{|T_j^\e(z_i-z_j)|^2}|}\Bigl).
\end{equation*}
and we would have to prove that the modified $v_i^{\e}$ behaves as in \eqref{v} because $T_j^\e(x)= \b_j^\e \dfrac{x}\e+h(\dfrac{x}\e)$ with $h'(x)=\mathcal{O}(\frac{1}{|x|^2})$ at infinity.

For the sake of simplicity, we will work with the disks (see \cite{BLM,ift_lop_euler,LM}, where an extension to more general domains was considered).
\end{remark}

To prove Proposition~\ref{B_n}, we write:
\begin{equation} \label{harm_decompo}
\frac{1}{2\pi}\frac{(x-z_i)^\perp}{|x-z_i|^2} - EX_i^\e=\Bigl(\frac{1}{2\pi}\frac{(x-z_i)^\perp}{|x-z_i|^2}-Ev_i^\e\Bigl)+E\Bigl(v_i^\e- X_i^\e\Bigl).
\end{equation}

\bigskip

{\bf Step 1: Convergence of the first term on the right hand side of \eqref{harm_decompo}.}

Let us fix $p\in [1,2)$. As this term is explicit, we can compute
\begin{eqnarray*}
\frac{1}{2\pi}\frac{(x-z_i)^\perp}{|x-z_i|^2}-Ev_i^\e(x)&=&(1-E)\frac{1}{2\pi}\frac{(x-z_i)^\perp}{|x-z_i|^2}\\
&& + \sum_{j\neq i} (\na^\perp \f_j)(x)\frac{1}{2\pi}\ln \frac{|x-z_i^{\ast j,\e}|}{|x-z_j|}\\
&&+ E\sum_{j\neq i} \f_j(x)\frac{1}{2\pi}\Bigl( \frac{(x-z_i^{\ast j,\e})^\perp}{|x-z_i^{\ast j,\e}|^2}-\frac{(x-z_j)^\perp}{|x-z_j|^2}\Bigl).
\end{eqnarray*}

This first term above tends to zero strongly in $L^p$ because the map $x\mapsto 1/|x|$ belongs to $L^q_{\loc}$ for $q\in (p,2)$ and the Lebesgue measure of the support of $1-E$ is equal to $\pi n \e^2$.

The second term is a bounded function, compactly supported, which tends to zero pointwise in $\R^2$, because $z_i^{\ast j,\e}\to z_{j}$ as $\varepsilon\to 0$. By the dominated convergence theorem, we conclude that it tends to zero in $L^p$.

Concerning the last term, we remark that it also tends to zero pointwise in $\R^2\setminus \Bigl( \cup_{j\neq i} \{ z_j\}\Bigl)$.
By the convexity of $g:\ t\mapsto |t|^p$ (for $p>1$), we use that
\[ g\left(\frac{\sum \l_i x_i}{\sum \l_i}\right) \leq \frac{\sum \l_i g(x_i)}{\sum \l_i}, \]
for $\l_i=1$ to compute
\begin{equation*}\begin{split}
\Bigl| E\sum_{j\neq i} \f_j(x)\frac{1}{2\pi}\Bigl( \frac{(x-z_i^{\ast j,\e})^\perp}{|x-z_i^{\ast j,\e}|^2}&-\frac{(x-z_j)^\perp}{|x-z_j|^2}\Bigl) \Bigl|^p \\
\leq& (2n-2)^{p-1}E\sum_{j\neq i}\frac{ \f_j^p(x)}{(2\pi)^p}\Bigl( \frac{1}{|x-z_i^{\ast j,\e}|^p}+\frac{1}{|x-z_j|^p}\Bigl)\\
\leq& (2n-2)^{p-1}\sum_{j\neq i}\frac{1}{(2\pi)^p}\Bigl( \frac{\b(|x-z_i^{\ast j,\e}|)}{|x-z_i^{\ast j,\e}|^p}+\frac{\b(|x-z_j|)}{|x-z_j|^p}\Bigl) =:g_{\e}(x)
\end{split}\end{equation*}
where $\b\in C^\infty(\R_+,\R_+)$ such that $\b(s)=1$ if $s<R$ and $\b(s)=0$ if $s>R+1$, with $R$ big enough (e.g. $R=\rho$ holds). We have that $g_{\e}\to g$ pointwise in $\R^2\setminus \Bigl(\cup_{j\neq i} \{ z_j\}\Bigl)$, where
\[g(x)= (2n-2)^{p-1}\sum_{j\neq i}\frac{2}{(2\pi)^p}\frac{\b(|x-z_j|)}{|x-z_j|^p}.\]
As $p<2$, we note that $g_{\e}$ and $g$ are integrable and
\[\int g_\e= \left(\frac{n-1}{\pi}\right)^p \int_{\R^2} \frac{\b(|x|)}{|x|^p}\, dx = \int g, \]
which allows us to apply the generalized dominated convergence theorem (see \cite{royden}) to get the convergence for the first term on the right hand side in \eqref{harm_decompo}:
\begin{equation}\label{harmonic first limit}
\frac{1}{2\pi}\frac{(x-z_i)^\perp}{|x-z_i|^2}-Ev_i^\e \to 0 \text{ strongly in } L^p(\R^2) \text{ when }\e \to 0.
\end{equation}

\bigskip

{\bf Step 2: convergence of the second term on the right hand side of \eqref{harm_decompo}.}

If we denote $\hat v_i^\e := v_i^\e-X_i^\e$, we deduce from Lemma~\ref{lem correction} that:
\begin{equation}\label{tilde_v}
\left\lbrace\begin{aligned}
&\div \hat v_i^\e = 0 &\text{ in } \Omega^\e\\
&\curl \hat v_i^\e = \hat \om^\e &\text{ in } \Omega^\e\\
&\hat v_i^\e \cdot \hat n =0 &\text{ on } \pd \Omega^\e\\
&\oint_{\pd B^\e_j} \hat v_i^\e \cdot \hat\t\, ds = 0 & \text{ for all }j=1\dots n
\end{aligned}\right.
\end{equation}
where
\[\hat \om^\e (x)= -\frac{1}{2\pi}\sum_{j\neq i} \left[ \D \f_j(x) \ln \frac{|x-z_i^{\ast j,\e}|}{|x-z_j|}+2\na \f_j(x)\cdot \Bigl( \frac{x-z_i^{\ast j,\e}}{|x-z_i^{\ast j,\e}|^2}-\frac{x-z_j}{|x-z_j|^2}\Bigl) \right].\]
Thanks to \eqref{Xi infini}, we note that $\hat v_i^\e(x)=\mathcal{O}(\frac1{|x|^2})$ at infinity, therefore,
\[ \int_{\Omega^\e} \hat \om^\e =-\sum_{j=1}^n \oint_{\partial \Omega^\e_j} \hat v_i^\e \cdot \hat\t\, ds + \lim_{R\to \infty} \oint_{\pd B(0,R)} \hat v_i^\e \cdot \hat\t\, ds= 0.\]

One of the main ideas in this proof is to perform an integration by parts in order to get the following.

\begin{lemma}\label{general_lemma}
Let $\hat v_i^\e$ be a vector field which verifies System \eqref{tilde_v} and such that $\int \hat \om^\e=0$. If we denote by $\p_i^\e$ any stream function of $\hat v_i^\e$ (i.e. $\hat v_i^\e = \na^\perp \p_i^\e$), then we have:
\[ \|\hat v_i^\e\|^2_{L^2(\Omega^\e)} \leq \left(\sup_{\supp(\hat \om^\e)} |\p_i^\e| \right)\ \|\hat \om^\e\|_{L^1(\Omega^\e)}.\]
\end{lemma}

\begin{proof}
We write the Biot-Savart law \eqref{biot} for $\hat v_i^\e$:
\[ \hat v_i^\e = K^{\e}[\hat \om^\e]+\sum_{j=1}^n \a_j[\hat \om^\e] X_j^{\e}\]
where $\a_j[\hat \om^\e] = \int_{\Omega^{\e}} \F_j^{\e} \hat\om^\e \, dx$. We can rewrite this equation with the stream functions:
\[ \p_i^\e = \P_D^{\e}[\hat \om^\e]+\sum_{j=1}^n \a_j[\hat \om^\e] \P_j^{\e}+ C,\]
with $C$ a constant. From \eqref{psi_prop} we recall that $\P_D^{\e}[\hat \om^\e]$ is bounded at infinity. Using the behavior at infinity of $\P_j^{\e}$ \eqref{Psi infini} we compute:
\[\sum_{j=1}^n \a_j[\hat \om^\e] \P_j^{\e}= \sum_{j=1}^n \a_j[\hat \om^\e] \Big(\frac{\ln |x|}{2\pi} + \mc{O}\Big(\frac{1}{|x|}\Big)\Big)= \frac{\ln |x|}{2\pi} \sum_{j=1}^n \int \F_j^{\e} \hat\om^\e + \mc{O}\Big(\frac{1}{|x|}\Big)=\mc{O}\Big(\frac{1}{|x|}\Big),\]
where we have used that $\sum_{j=1}^n \F_j^{\e}\equiv 1$ and $\int \hat \om^\e=0$. Then, we have obtained that $\p_i^\e$ is bounded at infinity. In the same way, we can prove that $\hat v_i^\e = \mc{O}(\frac{1}{|x|^2})$ at infinity. Recalling that the stream functions are constant on each boundary component, we infer from the vanishing circulations that
\begin{eqnarray*}
\|\hat v_i^\e\|_{L^2(\Omega^\e)}^2&=& \int_{\Omega^\e} \hat v_i^\e \cdot \hat v_i^\e\, dx= -\int_{\Omega^\e} \na \p_i^\e \cdot \hat v_i^{\e\perp} \, dx\\
&=& -\int_{\Omega^\e} \p_i^\e \hat\om^\e\, dx - \sum_{j=1}^n \oint_{\partial B_j^\e} \p_i^\e\ \hat v_i^\e\cdot\hat\t\, ds+ \lim_{R\to\infty} \oint_{\pd B(0,R)} \p_i^\e\ \hat v_i^\e\cdot\hat\t\, ds\\
&=& -\int_{\Omega^\e} \p_i^\e \hat\om^\e\, dx\\
&\leq & \left( \sup_{\supp(\hat \om^\e)} |\p_i^\e| \right) \ \|\hat \om^\e\|_{L^1(\Omega^\e)},
\end{eqnarray*}
which concludes the proof of the lemma.
\end{proof}

Using that the support of $\na \f_j$ is far from $z_j$, estimating $\|\hat \om^\e\|_{L^1}$ is not too difficult. Indeed, $x-z_i^{\ast j,\e}=x-z_j+\e^2\frac{z_i-z_j}{|z_i-z_j|^2}$ then
\begin{equation}\label{om_L1}
\|\hat \om^\e\|_{L^1(\Omega^\e)}\leq C\e^2.
\end{equation}

Now we need an estimate of $\p_i^\e$ on the support of $\hat \om^\e$, i.e. on the support of $\na \f_j$. For that, we use the explicit formula of $v_i^\e$ \eqref{v} to introduce the following stream function of $\hat v_i^\e = v_i^\e-X_i^\e$:
\[\p_i^\e(x) := \frac{1}{2\pi} \ln |x-z_i| - \sum_{j\neq i}\Bigl( \f_j\frac{1}{2\pi}\ln \frac{|x-z_i^{\ast j,\e}|}{|x-z_j|} \Bigl)-(\P_i^\e-c_{i,i}^\e),\]
where $c_{i,i}^\e$ is the value of $\P_i^\e$ on $\pd B_i^\e$. The first term on the right hand side is bounded by a constant and the second by $C\e^2$ (unifomly on the support of $\hat \om^\e$). The last term is the hardest to treat. The other main idea in this proof is to use the tools available for bounded domains, such as the maximum principle. To do this, we will use
\[
\mathbf{i}=\mathbf{i}(x):=x/|x|^2
\] the inversion with respect to the unit circle, which sends the exterior of the unit disk to the interior of the unit disk.
Without loss of generality, let us assume that the obstacle $B_i^\e$ is centered at the origin: $B_i^\e=B(0,\e)$. As the radius of $B_i^\e$ is $\e$, we should use
\[
\mathbf{i}_\e(x):=\mathbf{i}(x/\e)=\e \mathbf{i}(x).
\]
The image of $\Omega^\e$ by the last inversion, denoted by $\tilde \Omega^\e$, is the unit disk with $n-1$ holes. These holes $\tilde B^\e_j$ correspond to the image of $B_j^\e$ by $\mathbf{i}_\e$, for $j\neq i$. We give here some properties of the inversion $\mathbf{i}_\e$:
\begin{lemma}\label{disk_inversion}
We have that
\begin{itemize}
\item $\mathbf{i}_\e^{-1}(x) = \mathbf{i}_\e(x)=\e \mathbf{i}(x)$;
\item $\mathbf{i}_\e(B((d,0),\e))=B((\frac{\e}{d(1-\e^2)},0),\frac{\e^2}{d(1-\e^2)})$;
\item $\det D\mathbf{i}(x)=-1/|x|^4$.
\end{itemize}
Let $f:\Omega^{\e} \to\R$, and $g(x):= f(\mathbf{i}_\e^{-1}(x))$. If $f(x)=\mc{O}(1)$ at infinity and $\Delta f \equiv 0$ in a neighborhood of infinity, then
\begin{itemize}
\item $\na^\perp g (x) = -\e D\mathbf{i}(x)\ (\na^\perp f)(\e \mathbf{i}(x))$;
\item $\D g(x) = \e^2 |\det (D\mathbf{i}(x))|\ (\D f)(\e \mathbf{i}(x))$.
\end{itemize}
If $f(x)=\frac{\a}{2\pi}\ln|x| + \mc{O}(1)$ and $\D f = 0$, then $\D g(x) = - \a \d_0$.
\end{lemma}
The last item can be proved using that a harmonic function in $B(0,\r)\setminus\{0\}$ which is bounded can be extended to a harmonic function in $B(0,\r)$. All the other items can be shown by basic computations. Now, we use this lemma to prove the following

\begin{lemma} \label{lemma_inversion}
There exists a constant $C$, independent of $\e$, such that
\[ \sup_{\supp(\hat \om^\e)} |\P_i^\e-c_{i,i}^\e| \leq C(|\ln \e| + 1 + \sum_{j=1}^n \|\hat v_j^\e\|^2_{L^2(\Omega^\e)}).\]
\end{lemma}

\begin{proof}
By the definition of $\P_i^\e$, we know that there exists some constants $c_{i,j}^\e$ such that
\[
\P_i^\e \equiv c_{i,j}^\e \text{ on } \partial B_j^\e, \text{ for all } j=1\dots n.
\]
Let us denote the inversion of $\P_i^\e$ by $\tilde \P^\e(x):= \P_i^\e(\mathbf{i}_\e^{-1}(x))-c_{i,i}^\e=\P_i^\e(\e \mathbf{i}(x))-c_{i,i}^\e$. Then it verifies:
\begin{equation*}
\left\lbrace\begin{aligned}
&\D \tilde \P^\e =- \d_0 &\text{ in } \tilde \Omega^\e\\
&\tilde \P^\e=0 &\text{ on } \pd B(0,1)\\
&\tilde \P^\e=c_j^\e:=c_{i,j}^\e-c_{i,i}^\e &\text{ on } \pd \tilde B^\e_j \text{, for all } j\neq i\\
&\oint_{\pd B(0,1)} \na^\perp\tilde \P^\e \cdot \hat\t\, ds = -1 & \\
&\oint_{\pd \tilde B^\e_j} \na^\perp\tilde \P^\e \cdot \hat\t\, ds = 0 & \text{ for all }j=1\dots n, \ j\neq i,
\end{aligned}\right.
\end{equation*}
after remarking, by a simple calculation, that the circulation changes of sign by the inversion.

In bounded domain, we deduce from Section~\ref{sect : biot} the following decomposition:
\begin{equation}\label{dirichlet_decompo}
\tilde \P^\e = \tilde \P_D^\e[-\d_0]+ \sum_{j\neq i} c_j^\e \tilde \F_j^\e.
\end{equation}
We have already proved in Section~\ref{sect : biot} that $0\leq \tilde \F_j^\e \leq 1$.

If we denote $f:=\tilde \P_D^\e[-\d_0] + \frac{1}{2\pi} \ln |x|$, we have:
\[ \D f=0,\ f=0 \text{ on } \pd B(0,1), \text{ and } f= \frac{1}{2\pi} \ln |x| \text{ on }\pd \tilde B^\e_j.\]
Using that the distance between the origin and any hole is $\mc{O}(\e)$ (see Lemma~\ref{disk_inversion}), we can apply the maximum principle to $f$ to say that $|f(x)|\leq C |\ln{\e}|$. Then we have
\[|\tilde \P_D^\e[-\d_0](x)| \leq C(|\ln \e| + |\ln |x||)\]
which verifies:
\begin{equation*}
|\tilde \P_D^\e[-\d_0]| \leq C |\ln \e| \text{ on }\mathbf{i}_\e( \supp(\hat \om)).
\end{equation*}

To finish the estimate of $\tilde \P^\e$ in \eqref{dirichlet_decompo}, we need to estimate the constants $c_j^\e$. We compute the circulation of $\na^\perp\tilde \P^\e$
\[0 = \oint_{\pd \tilde B^\e_k} \na^\perp\tilde \P^\e \cdot \hat\t\, ds = \oint_{\pd \tilde B^\e_k} \na^\perp\tilde \P_D^\e[-\d_0] \cdot \hat\t\, ds + \sum_{j\neq i} c_j^\e \oint_{\pd \tilde B^\e_k} \na^\perp\tilde \F_j^\e \cdot \hat\t\, ds,\]
for any $k\neq i$. By \eqref{circu of K}, we compute the circulation of $\na^\perp \tilde \P_D^\e$
\[\oint_{\pd \tilde B^\e_k} \na^\perp\tilde \P_D^\e[-\d_0] \cdot \hat\t\, ds = -\int_{\tilde \Omega^\e} \tilde \F_k^\e \D \tilde \P_D^\e[-\d_0]=\tilde\F_k^\e(0).\]
If we denote by $P$ the matrix with coefficients $p_{k,j}=\oint_{\pd \tilde B^\e_k} \na^\perp\tilde \F_j^\e \cdot \hat\t\, ds=-\int_{\tilde \Omega^\e} \nabla\tilde \F_j^\e \cdot \nabla \tilde \F_k^\e$ for $k,j \neq i$, then we have
\[P\begin{pmatrix} c_1^\e \\ \vdots \\ c_n^\e \end{pmatrix}= - \begin{pmatrix} \tilde \F_1^\e(0) \\ \vdots \\ \tilde \F^\e_n(0) \end{pmatrix},\]
where we know that $|\tilde \F_j^\e| \leq 1$. Moreover, if we expand the harmonic vector field $\na^\perp\tilde \F_j^\e$ in the basis $\{\tilde \P_l^\e\}$, we get
\begin{eqnarray}
\na^\perp\tilde \F_j^\e&=& \sum_{l\neq i}\Bigl(\oint_{\pd \tilde B^\e_l} \na^\perp\tilde \F_j^\e \cdot \hat\t\, ds \Bigl) \tilde X_l^\e \nonumber\\
\int_{\tilde \Omega^\e} \na^\perp\tilde \F_j^\e \cdot \tilde X_k^\e &=& \sum_{l\neq i}\Bigl(\int_{\tilde \Omega^\e} \tilde X_k^\e \cdot \tilde X_l^\e \Bigl) \Bigl( \oint_{\pd \tilde B^\e_l} \na^\perp\tilde \F_j^\e \cdot \hat\t\, ds \Bigl).\label{phi_X}
\end{eqnarray}
Concerning the term on the left hand side, we integrate by parts:
\[\int_{\tilde \Omega^\e} \na^\perp\tilde \F_j^\e \cdot \tilde X_k^\e=-\int_{\tilde \Omega^\e} \tilde \F_j \curl \tilde X_k^\e-\sum_{l\neq i} \tilde \F_j^\e|_{\pd \tilde B^\e_l} \oint_{\pd \tilde B^\e_l} \tilde X_k^\e \cdot \hat\t\, ds =-\d_{j,k}.\]
Then, if we denote $M$ by the matrix $m_{l,k}=\int_{\tilde \Omega^\e} \tilde X_k^\e \cdot \tilde X_l^\e$, identity \eqref{phi_X} can be written:
\[-I_{n-1} = P\ M\]
which means that $P^{-1}=-M$ and that $|c_j^\e|\leq \sum_{k\neq i} |m_{k,j} |\leq C\sum_{k\neq i} \|\tilde X_k^\e\|_{L^2(\tilde \Omega^\e)}^2$.

The last step of this proof is to evaluate $\|\tilde X_k^\e\|_{L^2(\tilde \Omega^\e)}^2$. Actually, we remark that
\[\tilde \P_k^\e(x)=-(\P_k^\e-c_{k,i}^\e-\P_i^\e+c_{i,i}^\e)(\mathbf{i}_\e^{-1}(x)).\]
Using that
\[\tilde X_k^\e=\na^\perp\tilde \P_k^\e(x)= -D(\mathbf{i}_\e^{-1})^t(x)(X_k^\e-X_i^\e)(\mathbf{i}_\e^{-1}(x)),\]
by Lemma~\ref{disk_inversion}, we compute
\begin{eqnarray*}
\|\tilde X_k^\e\|_{L^2(\tilde \Omega^\e)}^2 &\leq & \int_{\tilde \Omega^\e} |D(\mathbf{i}_\e^{-1}(x))|^2 |X_k^\e-X_i^\e|^2(\mathbf{i}_\e^{-1}(x))\, dx\\
&\leq& \int_{\Omega^\e} |X_k^\e-X_i^\e|^2(y) \, dy\\
&\leq& \| X_k^\e-X_i^\e\|_{L^2(\Omega^\e)}^2.
\end{eqnarray*}
Recalling the definition of $\hat v_i^\e$, we write
\begin{eqnarray*}
\| X_k^\e-X_i^\e\|_{L^2(\Omega^\e)} &\leq & \|\hat v_k^\e\|_{L^2(\Omega^\e)} + \|\hat v_i^\e\|_{L^2(\Omega^\e)} \\
&&+\frac{1}{2\pi} \left\|\frac{(x-z_k)^\perp}{|x-z_k|^2} -\frac{(x-z_i)^\perp}{|x-z_i|^2}\right\|_{L^2(\Omega^\e)} \\
&& +\left\| \na^\perp \Bigl(\sum_{j\neq k} \f_j\frac{1}{2\pi}\ln \frac{|x-z_k^{\ast j,\e}|}{|x-z_j|}-\sum_{j\neq i} \f_j\frac{1}{2\pi}\ln \frac{|x-z_i^{\ast j,\e}|}{|x-z_j|} \Bigl)\right\|_{L^2(\Omega^\e)}.
\end{eqnarray*}
As $\int_\e^C\frac{1}{s}\, ds= \ln C - \ln \e$, the last term in the right hand side is bounded by $C(1+|\ln \e|)^{1/2}$. Adding that $\frac{(x-z_k)^\perp}{|x-z_k|^2} -\frac{(x-z_i)^\perp}{|x-z_i|^2}$ is square integrable at infinity, the third term in the right hand side is also bounded by $C(1+|\ln \e|)^{1/2}$.

This concludes the proof of the lemma.
\end{proof}

Putting together all the results of this subsection, we conclude that
\[ \|\hat v_i^\e\|^2_{L^2(\Omega^\e)} \leq C \Bigl(|\ln \e| + \sum_{j=1}^n \|\hat v_j^\e\|^2_{L^2(\Omega^\e)}\Bigl)\e^2, \]
which implies, by summing on $i$:
\begin{equation*}
 \sum_{i=1}^n\|\hat v_i^\e\|^2_{L^2(\Omega^\e)} \leq \frac{nC |\ln \e|\e^2}{1-nC\e^2}.
 \end{equation*}
In particular, this proves that $\|v_i^\e-X_i^\e\|_{L^2(\Omega^\e)}\to 0$. Bringing together with \eqref{harmonic first limit}, Proposition~\ref{B_n} is proved.

\subsection{The Laplacian-inverse (proof of Proposition~\ref{C_n})} \label{sect : C}\

As in the previous subsection, $k$ is fixed, so that to simplify notation, we will omit the parameter $k$ in all the functions and domains. Let $M_0$, and $\{ z_i\}_{i=1\dots n}\in (\R^2)^{n}$ be fixed. The goal of this section is to prove that the part with zero circulation around each $B_i^\e$ in the Biot-Savart law:
\[EK^{\e}[f\vert_{\Omega^{\e}}]+\sum_{i=1}^{n} m_i^{\e}[f] EX_i^{\e}\]
minus the Biot-Savart formula in $\R^2$:
\[K_{\R^2}[f\vert_{\Omega^{\e}}] \]
converges strongly in $L^2(\R^2)$ as $\e\to 0$, uniformly in $f$, for those $f$ verifying:
\[ \| f \|_{L^1\cap L^{p_0}(\R^2)} \leq M_{0}\quad \text{and} \quad f \text{ is compactly supported.}
\]

As before, we use the cut-off function $\f_j$ defined in \eqref{cut_off}, and the notation:
\begin{equation}\label{y ast}
y^{\ast j,\e}:= z_j+\e^2\frac{y-z_j}{|y-z_j|^2},
\end{equation}
 but now we introduce
\begin{equation*}
v^\e[f](x):= \frac{1}{2\pi} \int_{\Omega^\e} \frac{(x-y)^\perp}{|x-y|^2} f(y) \, dy -\na^\perp \Bigl(\sum_{i=1}^n \frac{\f_i(x)}{2\pi} \int_{\Omega^\e}\ln \frac{|x-y^{\ast i,\e}|}{|x-z_i|} f(y)\, dy \Bigl).
\end{equation*}
The properties of such a vector field are listed here:
\begin{lemma}\label{lem correction2}
For any $f\in L^{p_{0}}_{c}(\R^2)$ and $\e\in(0,\rho/2)$, we have that:
\begin{enumerate}
\item $v^\e[f]$ is divergence free ;
\item $v^\e[f]$ is tangent to the boundary ;
\item the circulations of $v^\e[f]$ around $B_i^\e$ are equal to zero, for all $i=1\dots n$ ;
\item the curl of $v^\e[f]$ is equal to:
\begin{equation*}
 f\vert_{\Omega^{\e}}(x) -\sum_{i=1}^n \frac{\D\f_i(x)}{2\pi} \int_{\Omega^\e}\ln \frac{|x-y^{\ast i,\e}|}{|x-z_i|} f(y) \, dy
 -\sum_{i=1}^n \frac{\na\f_i(x)}{\pi}\cdot \int_{\Omega^\e}\Bigl( \frac{x-y^{\ast i,\e}}{|x-y^{\ast i,\e}|^2}-\frac{x-z_i}{|x-z_i|^2} \Bigl) f(y) \, dy.
\end{equation*}
\end{enumerate}
\end{lemma}
\begin{proof}
Concerning items (1),(2) and (4), the proof is exactly the same as that of Lemma~\ref{lem correction}. The only small difference is that to prove (3). For any $y\in \OM^\e$ fixed, we can consider $r_{y}>0$ such that $B(y^{\ast i,\e},r_{y})\Subset B(z_{i},\varepsilon)$, and we deduce by Stokes formula that:
\begin{align*}
\frac{1}{2\pi}\oint_{\partial B_{i}^\e} \nabla^\perp\Bigl(\ln\frac{|x-y^{\ast i,\e}|}{|x-z_i|}\Bigl) \cdot \hat \tau \, ds =
&= \frac{1}{2\pi}\oint_{\partial B_{i}^\e}\frac{(x-z_i)^\perp}{|x-z_i|^2}\cdot \hat \tau \, ds - \frac{1}{2\pi}\oint_{\partial B(y^{\ast j,\e},r_{y})}\frac{(x-y^{\ast j,\e})^\perp}{|x-y^{\ast j,\e}|^2}\cdot \hat \tau \, ds\\
&= 1-1=0.
\end{align*}
Moreover, as $f\vert_{\Omega^{\e}}$ belongs to $L^1\cap L^{p_{0}}(\R^2)$ with $p_{0}>2$, we claim that $x\mapsto K_{\R^2}[f\vert_{\Omega^{\e}}](x)$ is continuous on $\R^2$. Indeed, by standard estimates, we know that it is bounded (see \cite{Dragos}, for example), and, by the Calder\'on-Zygmund inequality we may infer that its gradient belongs to $L^{p_{0}}$. So $K_{\R^2}[f\vert_{\Omega^{\e}}]$ belongs to $W^{1,p_{0}}_{\loc}(\R^2)$ which is embedded in $\mathcal{C}(\R^2)$. Moreover, $K_{\R^2}[f\vert_{\Omega^{\e}}]$ is curl free on $B_{i}^\e$, and therefore,
$$\frac{1}{2\pi}\oint_{\partial B_{i}^\e}K_{\R^2}[f\vert_{\Omega^{\e}}](x)\cdot \hat \tau \, ds =0,$$
which concludes the proof of the lemma.
\end{proof}

To prove Proposition~\ref{C_n}, we decompose as follows:
\[K_{\R^2}[f\vert_{\Omega^{\e}}]-E\Bigl(K^\e[f\vert_{\Omega^{\e}}]+\sum_{i=1}^n m_i^\e[f] X_i^\e\Bigl) =\Bigl(K_{\R^2}[f\vert_{\Omega^{\e}}]-Ev^\e[f]\Bigl)+E\Bigl(v^\e[f]-K^\e[f\vert_{\Omega^{\e}}]-\sum_{i=1}^n m_i^\e[f] X_i^\e \Bigl).\]

\bigskip

{\bf Step 1: Convergence of the first term on the right hand side.}

We compute
\begin{eqnarray*}
K_{\R^2}[f\vert_{\Omega^{\e}}]-Ev^\e[f]&=& (1-E) K_{\R^2}[f\vert_{\Omega^{\e}}]\\
&&+\sum_{i=1}^n \frac{\na^\perp \f_i(x)}{2\pi}\int_{\Omega^\e}\ln \frac{|x-y^{\ast i,\e}|}{|x-z_i|}f(y)\, dy\\
&&+ E\sum_{i=1}^n \frac{\f_i(x)}{2\pi} \int_{\Omega^\e}\Bigl(\frac{x-y^{\ast i,\e}}{|x-y^{\ast i,\e}|^2}-\frac{x-z_i}{|x-z_i|^2}\Bigl)^\perp f(y)\, dy.
\end{eqnarray*}

As $K_{\R^2}[f\vert_{\Omega^{\e}}]$ is uniformly bounded by $C\|f\|_{L^1(\R^2)}^{\alpha}\|f\|_{L^{p_{0}}(\R^2)}^{1-\alpha}\leq CM_0$ (with $\alpha \in (0,1)$ which depends only on $p_{0}$, see \cite{Dragos}, for example $\alpha_{\infty}=1/2$), it is obvious that the first term on the right hand side tends to zero in $L^2(\R^2)$ uniformly in $f$ verifying $\| f \|_{L^1\cap L^{p_0}(\R^2)} \leq M_{0}$.

For the second term, as $y^{\ast i,\e} \in B_{i}^\varepsilon$ and $x \in \supp \nabla \varphi_{i}$, we remark that
\begin{eqnarray*}
\Bigl| \sum_{i=1}^n \frac{\na^\perp \f_i(x)}{2\pi}\int_{\Omega^\e}\ln \frac{|x-y^{\ast i,\e}|}{|x-z_i|}f(y)\, dy \Bigl| &\leq&
C \varepsilon \sum_{i=1}^n \frac{|\na^\perp \f_i(x)|}{2\pi} \| f \|_{L^{1}(\R^2)}
\end{eqnarray*}
which tends to zero in $L^2(\R^2)$ uniformly in $f$ verifying $\| f \|_{L^1(\R^2)} \leq M_{0}$.

For the last term, we decompose the integrals in two parts:
\begin{multline*}
\Big|\frac{\f_i(x)}{2\pi} \int_{\Omega^\e}\Bigl(\frac{x-y^{\ast i,\e}}{|x-y^{\ast i,\e}|^2}-\frac{x-z_i}{|x-z_i|^2}\Bigl)^\perp f(y)\, dy\Big|
\leq \frac{|\f_i(x)|}{2\pi} \Big( \int_{B(z_{i},2\varepsilon)\setminus B_{i}^\varepsilon} \Big(\frac1{|x-y^{\ast i,\e}|}+\frac{1}{|x-z_i|}\Bigl) |f(y)|\, dy\\
+ \int_{\Omega^\varepsilon\setminus B(z_{i},2\varepsilon)} \Big| \frac{x-y^{\ast i,\e}}{|x-y^{\ast i,\e}|^2}-\frac{x-z_i}{|x-z_i|^2}\Bigl| |f(y)|\, dy\Big).
\end{multline*}
For the first integral, we can verify that $|x-y^{\ast i,\e}|^2 |y-z_{i}|^2 = |y-x^{\ast i,\e}|^2 |x-z_{i}|^2$, hence we have
\begin{align*}
 \int_{B(z_{i},2\varepsilon)\setminus B_{i}^\varepsilon} \Big(\frac1{|x-y^{\ast i,\e}|}+\frac{1}{|x-z_i|}\Bigl) |f(y)|\, dy
 &\leq
 \frac{1}{|x-z_i|} \Big( \int_{B(z_{i},2\varepsilon)\setminus B_{i}^\varepsilon} \frac{2\varepsilon}{|y-x^{\ast i,\e}|} |f(y)|\, dy + \|f \mathds{1}_{B(z_{i},2\varepsilon)}\|_{L^1}\big)\\
 &\leq \frac{1}{|x-z_i|} \Big( C \varepsilon\|f\|_{L^1(\R^2)}^{\alpha}\|f\|_{L^{p_{0}}(\R^2)}^{1-\alpha} + \|f\|_{L^{p_{0}}(\R^2)} (\pi 4\varepsilon^2)^{1/p_{0}'}\Big)\\
 &\leq \frac{CM_{0}\varepsilon }{|x-z_i|},
\end{align*}
where we have used again the estimate on $K_{\R^2}[f]$ \cite{Dragos}. Concerning the second integral, we use the relation $|\frac{a}{|a|^2}-\frac{b}{|b|^2}|=\frac{|a-b|}{|a||b|}$ and that for $y\in B(z_{i},2\varepsilon)^c$ and $x\in \Omega^\varepsilon$, we have $|x-y^{\ast i,\e}| \geq |x-z_{i}| - \frac{\varepsilon^2}{|y-z_{i}|} \geq \frac{\varepsilon}2$. So we compute:
\begin{align*}
 \int_{\Omega^\varepsilon\setminus B(z_{i},2\varepsilon)} \Big| \frac{x-y^{\ast i,\e}}{|x-y^{\ast i,\e}|^2}-\frac{x-z_i}{|x-z_i|^2}\Bigl| |f(y)|\, dy
&= \int_{\Omega^\varepsilon\setminus B(z_{i},2\varepsilon)} \frac{\varepsilon^2 |y-z_{i}|^{-1}}{|x-y^{\ast i,\e}||x-z_i|} |f(y)|\, dy\\
&\leq \frac{2\varepsilon}{|x-z_i|} \int_{\Omega^\varepsilon} \frac{|f(y)|}{|y-z_i|}\, dy\leq \frac{CM_{0}\varepsilon }{|x-z_i|}.
\end{align*}
Therefore, the $L^2(\R^2)$ norm of the last term can be estimated as
\begin{align*}
 \Big\| E\sum_{i=1}^n \frac{\f_i(x)}{2\pi} \int_{\Omega^\e}\Bigl(\frac{x-y^{\ast i,\e}}{|x-y^{\ast i,\e}|^2}-\frac{x-z_i}{|x-z_i|^2}\Bigl)^\perp f(y)\, dy \Big\|_{L^2}
&\leq \Big\| E\sum_{i=1}^n |\f_i(x)| \frac{CM_{0}\varepsilon }{|x-z_i|}\Big\|_{L^2}\\
&\leq CM_{0}\varepsilon n \Big\| \frac{1}{|z|}\Big\|_{L^2(B(0,\rho/2)\setminus B(0,\varepsilon))}\\
& \leq C M_{0} \varepsilon |\ln \varepsilon|,
\end{align*}
which also tends to zero as $\varepsilon\to 0$, uniformly in $f$ verifying $\| f \|_{L^1\cap L^{p_0}(\R^2)} \leq M_{0}$.

Therefore, we have established that $\|K_{\R^2}[f\vert_{\Omega^{\e}}]-Ev^\e[f] \|_{L^2(\R^2)}\to 0$ as $\varepsilon\to 0$, uniformly in $f$ verifying $\| f \|_{L^1\cap L^{p_0}(\R^2)} \leq M_{0}$.

\bigskip

{\bf Step 2: Convergence of the second term on the right hand side.}

We define $\hat v^\e[f]=v^\e[f]-K^\e[f\vert_{\Omega^{\e}}]-\sum_{i=1}^n m_i^\e[f] X_i^\e$. By Lemma~\ref{lem correction2}, this vector field verifies
\begin{equation*}
\left\lbrace\begin{aligned}
&\div \hat v^\e[f] = 0 = \div \Big(v^\e[f] -K_{\R^2}[f\vert_{\Omega^{\e}}]\Big) &\text{ in } \Omega^\e\\
&\curl \hat v^\e[f] =\curl v^\e[f] -f = \curl \Big(v^\e[f] -K_{\R^2}[f\vert_{\Omega^{\e}}]\Big) &\text{ in } \Omega^\e\\
&\oint_{B^\e_i} \hat v^\e[f] \cdot \hat\t\, ds = 0 =\oint_{B^\e_i} \Big(v^\e[f] -K_{\R^2}[f\vert_{\Omega^{\e}}]\Big)\cdot \hat\t\, ds& \text{ for all }i=1\dots n\\
& \lim \hat v^\e[f] = 0 =\lim \Big(v^\e[f] -K_{\R^2}[f\vert_{\Omega^{\e}}]\Big) &\text{ as } |x|\to \infty\\
&\hat v^\e[f] \cdot n =0 &\text{ on } \pd \Omega^\e.
\end{aligned}\right.
\end{equation*}
This implies that $\hat v^\e[f]$ is the Leray projection\footnote{projection on divergence free vector fields which are tangent to the boundary.} of $v^\e[f] -K_{\R^2}[f\vert_{\Omega^{\e}}]$. Therefore, by orthogonality of this projection in $L^2$, we have
\[
\| \hat v^\e[f] \|_{L^2(\Omega^\varepsilon)} \leq \| v^\e[f] -K_{\R^2}[f\vert_{\Omega^{\e}}] \|_{L^2(\Omega^\varepsilon)} \leq C M_{0} \varepsilon |\ln \varepsilon|.
\]

This concludes the proof of Proposition~\ref{C_n}.

\begin{remark}
 This argument is already present in \cite{BLM}, and it explains why $L^2$ plays a special role: for every $\varepsilon$, the Leray projector is an operator in $L^2$ bounded by $1$. This argument cannot be used for the harmonic part, because we only proved estimates in $L^p$ for $p<2$, and it is not clear that the Leray projection is uniformly continuous in $L^p$ (see \cite[Section 5.3]{LM}). The Step 2 in Section~\ref{sect : B} avoids such a consideration.
\end{remark}

\subsection{Construction of $\varepsilon_k$} \

The goal of this section is to construct suitable $\varepsilon_k$, so that Theorem~\ref{main} will hold true.

So let us fix $\om_0\in L^{p_0}_{c}(\R^2)$ for some $p_0 \in (2,\infty]$. Let us also fix $R_0>0$ so that $\supp \omega_{0}\subset B(0,R_{0})$. We consider a sequence $\om_{0}^k \in C^\infty_c(B(0,R_{0}))$, $\{ \g_i^{k}\}_{i=1\dots n_k}\subset \R$, $\{ z_i^{k}\}_{i=1\dots n_k}\subset B(0,R_{0})$ such that
\[\om_{0}^k \rightharpoonup \om_{0} \text{ weakly in } L^{p_{0}}(\R^2)\]
and
\[ \sum_{i=1}^{n_k} \g_i^{k} \d_{z_i^k} \rightharpoonup \mu \text{ weak-$*$ in }\ \M(\R^2),\]
for some $\mu\in \M_c(\R^2)$. The first limit allows us to define $M_{0}\in \R_+$ as following:
\begin{equation}\label{defi M0}
M_{0}:= \sup_{k\in \N} \{ \| \om_{0}^k \|_{L^1\cap L^{p_{0}}(\R^2)}\}.
\end{equation}
The second limit implies that $\| \sum_{i=1}^{n_k} \g_i^{k} \d_{z_i^k} \|_{\M(\R^2)}$ is uniformly bounded in $k$, so there exists $M_{1} \in \R_+$ such that:
\begin{equation}\label{defi gM}
|\g_i^{k}| \leq M_{1}, \quad \forall k\in \N, \ \forall i=1\dots n_{k}.
\end{equation}

Now, we fix $k\in \N^*$, and we are looking for a definition of $\varepsilon_k$.

By Proposition~\ref{B_n}, there exists $\e_{1}^k$ such that for any $\e\in (0,\e_{1}^k)$, there is $v_{i}^\e$ such that
\[
\Bigl\| \frac{(x-z_i)^\perp}{2\pi |x-z_i|^2} -Ev_{i}^\e(x) \Bigl\|_{L^1\cap L^{2-\frac1k}(\R^2)}+\|Ev_{i}^\e -EX_i^{\e,k} \Bigl\|_{L^2(\R^2)}\leq \frac{1}{n_{k}M_{1} k},\]
for any $i=1\dots n_k$.

By Proposition~\ref{C_n}, with $M_{0}$ defined in \eqref{defi M0}, there exists $\e_{2}^k$ such that for any $\e\in (0,\e_{2}^k)$ and $f$ verifying:
\[ \| f \|_{L^1\cap L^{p_0}(\R^2)} \leq M_{0} \quad \text{and} \quad f \text{ is compactly supported,}
\]
we have
\[
\Bigl\| K_{\R^2}[f\vert_{\Omega^{\e,k}}]-EK^{\e,k}[f\vert_{\Omega^{\e,k}}]-\sum_{i=1}^{n_k} m_i^{\e,k}[f] EX_i^{\e,k}\Bigl\|_{L^{2}(\R^2)} \leq \frac{1}{k}.
\]

We recall that $\e_{1}^k$ and $\e_{2}^k$ are chosen small enough such that the disks are disjoints. We finally choose:
\begin{equation}\label{defi ek}
\varepsilon_k:=\min\{ \e_{1}^k,\e_{2}^k,\frac1k \}.
\end{equation}


\section{Time evolution}\label{sect time}

The goal of this section is to prove Theorem~\ref{main}, using Propositions~\ref{A_n}, \ref{B_n} and \ref{C_n}.

So let us fix $\om_0\in L^{p_0}_{c}(\R^2)$ for some $p_0 \in (2,\infty]$. Let us also fix $R_0>0$ so that $\supp \omega_{0}\subset B(0,R_{0})$. We consider a sequence $\om_{0}^k \in C^\infty_c(B(0,R_{0}))$, $\{ \g_i^{k}\}_{i=1\dots n_k}\subset \R$, $\{ z_i^{k}\}_{i=1\dots n_k}\subset B(0,R_{0})$ such that
\[\om_{0}^k \rightharpoonup \om_{0} \text{ weakly in } L^{p_{0}}(\R^2)\]
and
\[ \sum_{i=1}^{n_k} \g_i^{k} \d_{z_i^k} \rightharpoonup \mu \text{ weak-$*$ in }\ \M(\R^2),\]
for some $\mu\in \M_c(\R^2)$.
As in the previous subsection, we introduce:
\begin{equation*}
M_{0}:= \sup_{k\in \N} \{ \| \om_{0}^k \|_{L^1\cap L^{p_{0}}(\R^2)}\}.
\end{equation*}
and $M_{1} \in \R_+$ by
\begin{equation*}
M_{1}:= \sup_{k\in \N,\ i=1\dots n_{k}} \{ |\g_i^{k}| \}.
\end{equation*}
For the sequence of radii $\varepsilon_k$ chosen in \eqref{defi ek}, we consider the domains
\[\Omega^{k}:= \R^2\setminus \Bigl(\overline{\bigcup_{i=1}^{n_k} B(z_i^k,\varepsilon_k)}\Bigl).\]

\subsection{Euler equations and vorticity estimates}\

As mentioned in Subsection~\ref{sect BS}, there exists a unique smooth vector field $u_0^k$ verifying \eqref{ini elliptic}. For such initial data, Kikuchi \cite{kiku} established existence and uniqueness of a global strong solution $u^k$ of \eqref{euler-velocity} in $\OM^k$. Moreover, he proved that this solution verifies \eqref{vort eq} in the sense of distributions. Thanks to the regularity of $u^k$, the method of characteristics implies that
\begin{itemize}
\item $\om^k(t,\cdot)$ is compactly supported for any $t \in \R_+$ (not uniformly);
 \item $\int_{\Omega^k}\omega^k(t,\cdot)$ is a conserved quantity;
 \item the $L^q$ norms of the vorticity are conserved for any $q\in [1,\infty]$.
\end{itemize}
Moreover, by Kelvin's Circulation Theorem,
\begin{itemize}
\item the circulation around each disk $B(z_{i}^k,\varepsilon_k)$ is conserved:
$$\oint_{\partial B_i^k} u^k(t,\cdot) \cdot \hat\t ds= \g_i^k,\ \forall t\geq 0.$$
\end{itemize}

Therefore, we infer that
\begin{equation}\label{vorticity est}
\| E \om^k(t,\cdot) \|_{L^1\cap L^{p_{0}}(\R^2)} = \| E\om^k_{0} \|_{L^1\cap L^{p_{0}}(\R^2)} \leq M_{0},\quad \forall t\in \R_+,\ \forall k\in \N^*.
\end{equation}

\subsection{Velocity estimates}\label{sect : estimates}\

We use the Biot-Savart law:
\begin{equation*}
u^k(t,x)=K^k[\om^k(t,\cdot)](x)+\sum_{i=1}^{n_{k}} \a_i^k[\om^k(t,\cdot)] X_i^k(x) \text{ with }\a_i^k[\om^k(t,\cdot)] = \int_{\Omega^k} \F_i^k(y) \om^k(t,y) \, dy +\g_i^k,
\end{equation*}
 together with Propositions~\ref{A_n}, \ref{B_n} and \ref{C_n} in order to get $L^p$ estimates for the velocity on compact sets.

\begin{proposition} \label{velocity estimate} With the above definitions, $E u^k$ is uniformly bounded in $L^{\infty}(\R_+, L^p_{\loc}(\R^2))$, for any $p\in [1,2)$. More precisely, for any $p\in (1,2)$, we can decompose the velocity as $E u^k=u_{1}^k+u_{2}^k+u_{3}^k$ such that
\[
\| u_{1}^k \|_{L^{\infty}(\R_+, L^p(\R^2))}, \ \| u_{2}^k \|_{L^{\infty}(\R_+, L^2(\R^2))},\ \| u_{3}^k \|_{L^{\infty}(\R_+, L^{p^*}(\R^2))} \text{ are uniformly bounded,}
\]
with
$p^*=(\frac1p-\frac12)^{-1}\in (2,\infty)$.
\end{proposition}

\begin{proof} Fix $p\in [1,2)$. We choose $k_1$, such that, for any $k\geq k_1$, we have:
\begin{eqnarray}
2-\frac1k\geq p\label{k2 cond2}.
\end{eqnarray}

We compare $Eu^k$ with
\[\mathbf{w}^k(t,x)=K_{\R^2}[E\om^k(t,\cdot)+\mu](x),\]
by decomposing
\begin{eqnarray*}
(\mathbf{w}^k-Eu^k)(t,x) &=& \Bigl( K_{\R^2}[\mu]-K_{\R^2}[\sum_{i=1}^{n_k} \g_i^k \d_{z_i^k}]\Bigl)(x) \\
&&+\sum_{i=1}^{n_k} \g_i^{k }\Bigl( \frac{(x-z_i^k)^\perp}{2\pi |x-z_i^k|^2} - EX_i^k(x)\Bigl)\\
&&+ \Bigl( K_{\R^2}[E\om^k(t,\cdot)]-EK^k[\om^k(t,\cdot)]-\sum_{i=1}^{n_k} m_i^k[\om^k(t,\cdot)] EX_i^k\Bigl)(x)\\
&=:& A_k(x) + B_k(x) + C_k[\om^k(t,\cdot)](x).
\end{eqnarray*}

By Proposition~\ref{A_n}, we have that
\[
\| A_{k,1} \|_{L^p(\R^2)},\ \| A_{k,2} \|_{L^{p^*}(\R^2)}\text{ are uniformly bounded},
\]
with $A_{k}=A_{k,1}+A_{k,2}$.

Thanks to \eqref{k2 cond2}, we can interpolate $L^p$ between $L^1$ and $L^{2-\frac1k}$; hence we deduce from the definition of $\e_{k}$ \eqref{defi ek} that
\[
\| B_{k,1} \|_{L^p(\R^2)} + \| B_{k,2} \|_{L^{2}(\R^2)} \leq \frac1k,
\]
with $B_{k,1}:=\sum_{i=1}^{n_k} \g_i^{k }( \frac{(x-z_i^k)^\perp}{2\pi |x-z_i^k|^2} - E v_i^k(x))$ and $B_{k,2}:=\sum_{i=1}^{n_k} \g_i^{k }( E v_i^k(x))-EX_{i}^k)$.

We already know from \eqref{vorticity est} that
\[ \| \ E\om^{k}(t,\cdot)\|_{L^1\cap L^{p_0}(\R^2)} \leq M_0, \forall t, \ \forall k.\]
As $E\om^{k}$ is compactly supported for all $t,k$, we deduce from the definition of $\e_{k}$ that
\[
\| C_k[\om^k(t,\cdot)] \|_{L^{2}(\R^2)} \leq \frac1k, \forall t, \ \forall k.
\]

By Hardy-Littlewood-Sobolev Theorem, we have
\[ \|K_{\R^2}[E\om^k(t,\cdot)] \|_{L^{p^*}(\R^2)} \leq C_{p}\|E\om^k(t,\cdot) \|_{L^{p}(\R^2)} \leq C_{p} M_{0}, \quad \forall t\in \R_+, \ \forall k\in \N^*,\]
hence, $\| K_{\R^2}[E\om^k(t,\cdot)] \|_{L^\infty(\R_+;L^{p^*}(\R^2))} $ is uniformly bounded.

It is clear from Subsection~\ref{number_obstacle} (which is the proof of Proposition~\ref{A_n}) that
$K_{\R^2}[\mu]$ is the sum of a function belonging to $L^{p}(\R^2)$ and a function belonging to $L^{p^*}(\R^2)$.

Finally, we can put together all the estimates to obtain that
\[
Eu^k(t,x)=K_{\R^2}[\mu](x)+K_{\R^2}[E\om^k(t,\cdot)](x) -A_k(x) - B_k(x) - C_k[\om^k(t,\cdot)](x)
\]
verifies the statement of the proposition.
\end{proof}

\subsection{Strong convergence for the velocities}\

First, we use standard compactness argument to extract a subsequence such that $K_{\R^2}[E\om^k(t,\cdot)]$ converges strongly.

\begin{proposition}\label{comp-vort}
There exists a subsequence, still denoted by $k$, and a function $\om \in L^\infty(\R_+; L^1\cap L^{p_0}(\R^2))$, such that
\begin{itemize}
\item ${\om}^{k} \rightharpoonup {\om}$ weak $*$ in $L^\infty(\R_+; L^{q}(\R^2))$ for any $q\in [1,p_{0}]$;
\item $K_{\R^2}[E\om^k] \to K_{\R^2}[\om]$ strongly in $L^2_{\loc}(\R_+\times\R^2)$.
\end{itemize}
\end{proposition}

\begin{proof}
From \eqref{vorticity est} and by the Banach-Alaoglu's theorem, we can extract a subsequence which converges weak $*$ in $L^\infty(\R_+; L^{1}\cap L^{p_{0}}(\R^2))$, which gives the first item.

We consider $p\in (1,2)$ such that $p'=(1+p^{-1})^{-1}\in (2,p_{0}]$ (we can take $p=p_{0}'$ if $p_{0}<\infty$, and $p=3/2$ otherwise). Let fix $T>0$, and we set $X= L^{3/2}(\R^2)$ and $Y=H^{-3}(\R^2)$. Then $X$ is a separable reflexive Banach space and $Y$ is a Banach space such that $X\hookrightarrow Y$ and $Y'$ is separable and dense in $X'$. By \eqref{vorticity est}, we already know that $\{ E \om^k \}$ is a bounded sequence in $L^\infty(0,T; X)$. Moreover, for any function $\F \in Y'=H^3(\R^2)$, and any $t\in [0,T]$, we use the equation verified by $\omega^k$ to get
\begin{eqnarray*}
(\pd_t E\om^k,\F)&=& -\int_{\OM^k} \div( u^k \om^k) \F =\int_{\OM^k} \om^k u^k \cdot \nabla \F \\
&\leq& (\| \omega^k \|_{L^{p'}} \| u^k_{1} \|_{L^{p}}+\| \omega^k \|_{L^{2}} \| u^k_{2} \|_{L^{2}}+\| \omega^k \|_{L^{(p^*)'}} \| u^k_{3} \|_{L^{p^*}} ) \| \nabla \F \|_{L^\infty}\\
& \leq& C \| \F \|_{H^3}.
\end{eqnarray*}
where we have used Proposition~\ref{velocity estimate} together with \eqref{vorticity est}. This means that $ E \om^k \in C([0,T];Y)$ and that for all $\F \in Y'$, $(E\om^k,\F)_{Y\times Y'}$ is uniformly continuous in $[0,T]$, uniformly in $k$. Therefore, Lemma C.1 of \cite{PLL} states that $\{ E \om^k \}$ is relatively compact in $C([0,T]; X-w)$.

Fixing $p_{1} \in (2,p_{0})$, this argument also holds with $X=L^{p_{1}}(\R^2)$. Therefore, we can extract a subsequence such that
\[
E\om^k \to \omega \text{ in } C([0,T]; L^{3/2}\cap L^{p_{1}}(\R^2)-w).
\]

Now, we prove that this implies that the function
\[
(t,x)\mapsto K_{\R^2}[(\om-E\om^k)(t,\cdot)](x)
\]
tends to zero in $L^2([0,T];L^{2}(K))$ when $k\to \infty$, for any $K$ compact subset of $\R^2$.

For $t\in [0,T]$ and $x\in K$ fixed, the map $y\mapsto \frac{(x-y)^\perp}{|x-y|^2}\mathds{1}_{\{y\in B(x,1)\}}$ belongs to $L^{p_{1}'}(\R^2)$ whereas $y\mapsto \frac{(x-y)^\perp}{|x-y|^2}\mathds{1}_{\{y\in B(x,1)^c\}}$ belongs to $L^{3}(\R^2)$, hence
\[
|\int_{\R^2} \frac{(x-y)^\perp}{|x-y|^2} (E\om^k-\om)(t,y)\, dy |\to 0 \text{ when } k\to\infty.
\]
Moreover, using the estimate of \cite{Dragos}, we have the uniform bound:
\[
|\int_{\R^2} \frac{(x-y)^\perp}{|x-y|^2} (E\om^k-\om)(t,y)\, dy |\leq C 2 M_{0}, \ \forall t\in [0,T],\ \forall x\in \R^2,\ \forall k\in \N^*.
\]
Hence, applying twice the dominated convergence theorem, we obtain the convergence of $K_{\R^2}[E\om^k-\om]$ in $L^2([0,T];L^{2}(K))$.

By a diagonal extraction on $T=N\in \N$, we can choose a subsequence holding for all $T$.
\end{proof}

The first item of the above proposition gives the assertion (b) in Theorem~\ref{main}.

The main result of this subsection is the following.

\begin{theorem}\label{theo : velocity} With the sequence $\{\om^k\}$ constructing in Proposition~\ref{comp-vort}, we have
that for any $p\in [1,2)$:
\[ Eu^k\longrightarrow u \text{ strongly in } L^2_{\loc}(\R_+; L^p_{\loc}(\R^2)) \text{ when } k\to \infty,\]
with $u=K_{\R^2}[\om+\mu]$.
\end{theorem}

In particular, this result implies part (a) of Theorem~\ref{main}.

\begin{proof} Let us fix $p\in [1,2)$, $T>0$ and $K$ a compact subset of $\R^2$. We also fix $\rho>0$ and we are looking for $k_{\rho}$ such that for all $k\geq k_{\rho}$ we have:
\[
\| u-Eu^k \|_{L^2([0,T]; L^p(K))}\leq \rho.
\]
First, we note that there exists $k_1$, such that for any $k\geq k_1$, we have:
\begin{eqnarray}
2-\frac1k\geq p.\label{k3 cond2}
\end{eqnarray}

As usual, we write:
\begin{eqnarray*}
(u-Eu^k)(t,x) &=& \Bigl( K_{\R^2}[\mu]-K_{\R^2}[\sum_{i=1}^{n_k} \g_i^k \d_{z_i^k}]\Bigl)(x) \\
&&+\sum_{i=1}^{n_k} \g_i^{k }\Bigl( \frac{(x-z_i^k)^\perp}{2\pi |x-z_i^k|^2} - EX_i^k(x)\Bigl)\\
&&+ \Bigl( K_{\R^2}[E\om^k(t,\cdot)]-EK^k[\om^k(t,\cdot)]-\sum_{i=1}^{n_k} m_i^k[\om^k(t,\cdot)] EX_i^k\Bigl)(x)\\
&&+K_{\R^2}[(\om-E\om^k)(t,\cdot)](x)\\
&=:& A_k(x) + B_k(x) + C_k[\om^k(t,\cdot)](x)+K_{\R^2}[(\om-E\om^k)(t,\cdot)](x).
\end{eqnarray*}

By Proposition~\ref{A_n}, there exists $k_{2}$ such that for any $k\geq k_2$, we have $\| A_{k} \|_{L^2([0,T];L^{p}(K))}\leq \rho/4.$
By the definition of $\e_{k}$ (see \eqref{defi ek}) together with \eqref{k3 cond2}, we see that there exists $k_{3}\geq k_{1}$ such that for any $k\geq k_3$
\[
\| B_{k} \|_{L^2([0,T];L^{p}(K))} \leq \rho/4.
\]

Moreover, we already know from \eqref{vorticity est} that
$\| \ E\om^{k}(t,\cdot)\|_{L^1\cap L^{p_0}(\R^2)} \leq M_0,$ for every $t\in[0,T]$, and all $k$,
so we deduce from the definition of $\e_{k}$ that there exists $k_{4}$ such that for any $k\geq k_4$, we have:
\[
\| C_k[\om^k(t,\cdot)] \|_{L^2([0,T];L^{p}(K))}\leq \rho/4.
\]

By Proposition~\ref{comp-vort}, it also clear that there exists $k_{5}$ such that for any $k\geq k_5$, we have:
\[
\| K_{\R^2}[(\om-E\om^k)(t,\cdot)](x)\|_{L^2([0,T];L^{p}(K))}\leq \rho/4.
\]

Denoting $k_{\rho}=\max_{i=1\dots 5}\{k_{i}\}$, we have proved that for any $k\geq k_{\rho}$
\[
\| u-Eu^k \|_{L^2([0,T]; L^p(K))}\leq \rho,
\]
which concludes the proof.
\end{proof}

\subsection{Passing to the limit in the Euler equations}\

We have obtained the convergence of the velocity and of the vorticity as required in points (a)-(b) of Theorem~\ref{main}.

The purpose of the rest of this section is to prove the point (c), namely that $u$ and $\om$ verify, in an appropriate sense, the system:
\begin{equation*}
\left\lbrace\begin{aligned}
&\pd_t {\om}+u\cdot \na {\om}=0 &\text{ in } (0,\infty) \times \R^2 \\
&\div u =0 &\text{ in } (0,\infty) \times \R^2 \\
&\curl u ={\om} + \mu &\text{ in } (0,\infty) \times \R^2 \\
&\lim_{|x|\to\infty}|u|=0 & \text{ for }t\in[0,\infty)\\
&{\om}(0,x)={\om}_0(x) &\text{ in } \R^2
\end{aligned}\right.
\end{equation*}

\begin{definition} The pair $(u,\om)$ is a weak solution of the previous system if
\begin{itemize}
\item[(a)] for any test function $\f\in C^\infty_c([0,\infty)\times\R^2)$ we have
$$\int_0^\infty\int_{\R^2}\f_t\om \ dxdt +\int_0^\infty \int_{\R^2}\na\f\cdot u\om \ dxdt+\int_{\R^2}\f(0,x)\om_0(x)\ dx=0,$$
\item[(b)] we have $\div u=0$ and $\curl u=\om+\mu$ in the sense of distributions of $\R^2$, with $|u|\to 0$ at infinity.
\end{itemize}
\end{definition}

\begin{theorem} The pair $(u,\om)$ obtained is a weak solution of the previous system.
\end{theorem}
\begin{proof} The second point of the definition is directly verified thanks to the explicit form of $u$:
\[ u(t,x) = K_{\R^2}[\om(t,\cdot)+\mu](x),\]
because $\mu$ is compactly supported, and $ K_{\R^2}[\om(t,\cdot)]$ belongs to $L^\infty(\R_+;W^{1,p}(\R^2))$ for some $p\in (2,p_{0})$ (by Hardy-Littlewood-Sobolev Theorem and Calder\'on-Zygmund inequality).

For any test function $\f\in C^\infty_0([0,\infty)\times \R^2)$, we have that $(u^k,\om^k)$ verifies \eqref{vort eq}:
\begin{eqnarray*}
\int_0^\infty\int_{\OM^k} (\om^k \partial_t \f + \om^k u^k \cdot \nabla \f) (t,x) \, dxdt + \int_{\OM^k}\f(0,x)\om_0^k(x)\ dx&=&0\\
\int_0^\infty\int_{\R^2} (E\om^k \partial_t \f + E\om^k Eu^k \cdot \nabla \f) (t,x) \, dxdt + \int_{\R^2}\f(0,x)\om_0^k(x) \mathds{1}_{\OM^k}\ dx&=&0.
\end{eqnarray*}
Therefore, we can pass easily to the limit thanks to the weak-strong convergence of the pair vorticity-velocity:
\[\int_0^\infty\int_{\R^2} (\om \partial_t \f + \om u \cdot \nabla \f) (t,x) \, dxdt + \int_{\R^2}\f(0,x)\om_0(x) \ dx=0,\]
which concludes the proof of the main theorem.
\end{proof}

\section{Final remarks and comments}\label{sect final}

\subsection{Fixed Dirac masses}\

As mentioned in the introduction, one application of our result is the approximation of an arbitrary compactly supported bounded Radon measure $\mu \in \mc{BM}_c(\R^2)$ by a square
grid discretization $\sum \g_i^k \d_{z_i^k}\rightharpoonup \mu$, where $z_i^k$ is the center of a square $\mc{C}_i^k$ (the side of the square has length $1/k$) and $\g_i^k:=\int_{\mc{C}_i^k} \varphi_{i}^k\, d\mu$ for $(\varphi_{i}^k)$ a partition of unity (with $\supp \varphi_{i}^k\subset z_{i}^k+[-1/k,1/k]^2$).

However, let us present another interesting example: the case where $\mu$ is a finite sum of Dirac masses.

For $\mu = \g \d_0$, we can consider the following setting:
\[ n_k=1, \ \Omega^{k}= \R^2\setminus \overline{B(0,\e_k)} \text{ and } \g^k_1=\g.\]
In this case, Theorem~\ref{main} is precisely the main result in \cite{ift_lop_euler}. Actually, by the uniqueness result of \cite{lac_miot}, we claim that for initial vorticity constant in a small neighborhood of $0$ (bounded and compactly supported), the limit holds for any sequence $\e_k\to 0$, without extraction of a subsequence.

\bigskip

More generally, if $\mu$ is a finite sum of Dirac: $\mu=\sum_{i=1}^n \g_i \d_{z_i}$, we can choose:
\[ n_k=n, \ \Omega^{k}:= \R^2\setminus (\cup_{i=1}^n \overline{B(z_i,\e_k)}) \text{ and } \g_i^k=\g_i,\]
for $\e_k< \inf_{i\neq j} |z_i-z_j|$. Therefore, we have proved that for any sequence $\e_k \to 0$, we can extract a subsequence such that the conclusion of Theorem~\ref{main} holds true, which is an extension of \cite{ift_lop_euler} to several disks. In the same way, \cite{lac_miot} states that, in the case of initial vorticity constant around each $z_i$ and belonging to $L^\infty_c(\R^2)$, the limit system has at most one solution, which implies that the convergences hold for all the sequence, without extracting a subsequence.

\bigskip

Concerning Dirac masses, another interesting consequence is the fusion of two Dirac masses. Considering $z_1^k$ and $z_2^k$ two sequences converging to the same point $z_0$, we note that
\[ \g_1 \d_{z_1^k} +\g_2 \d_{z_2^k} \rightharpoonup (\g_1+\g_2) \d_{z_0}.\]
Therefore, we can apply our theorem in this setting.

\subsection{Local uniqueness}\

Let us make a short comment about the uniqueness in case of vorticity compactly supported outside the measure $\mu$. In \cite{lac_miot}, the key to prove global uniqueness is to show that the vorticity never meets the support of $\mu$. Such an estimate appears to be challenging for our system, but we can already infer that we have a local uniqueness result for our system if $\om_0\in L^\infty_c(\R^2\setminus \supp \mu)$. Indeed, far from the support of $\mu$, the velocity is bounded, and as the vorticity is transported by the velocity, we state that there exists a time $T_0>0$ such that $\om(t,\cdot)$ is supported outside $\supp \mu$ for all $t\in [0,T_0]$. Then following \cite[Sect. 3]{lac_miot} we prove easily that the solution of the limit system (see point c) in Theorem~\ref{main}) is unique up to the time $T_0$.

\subsection{Flow around a curve}\label{sect:curve}\

Let $\Gamma: \Gamma(s),0\leq s\leq 1$ be a $C^2$-Jordan arc. The first author has proved in \cite{lac_euler} the existence of a global solution in the exterior of a curve. This solution is obtained by a compactness argument, on the solution in the exterior of a smooth thin obstacle shrinking to the curve. Actually, a formulation in the full plane was found: for $\om_0 \in L^\infty_c(\R^2\setminus \partial \Omega)$ and $\g\in \R$ given, there exists a pair $(u,\om)$
\[
u \in L^{\infty}_{\loc}(\R_+;L^2_{\loc}(\R^2)), \quad \om \in L^\infty(\R_+\times \R^2)
\]
verifying, in the sense of distributions, the system:
\begin{equation*}
\left\lbrace\begin{aligned}
&\pd_t \om+u\cdot \na \om=0 &\text{ in } \R^2\times (0,\infty) \\
&\div u =0 &\text{ in } \R^2\times (0,\infty) \\
&\curl u =\om +g_{\om}(s) \d_\Gamma &\text{ in } \R^2\times (0,\infty) \\
&\om(x,0)=\om_0(x) &\text{ in } \R^2
\end{aligned}\right.
\end{equation*}
where $\d_\Gamma$ is the Dirac delta on $\Gamma$, and $g_{\om}$ is explicitly given in terms of $\om$, $\g$ (which can be viewed as the initial circulation) and the shape of $\Gamma$. In fact $u$ is a vector field which is tangent to $\Gamma$ (with different values on each side of the curve), vanishing at infinity, with circulation around the curve $\Gamma$ equal to $\g$. This velocity is blowing up at the endpoints of the curve $\Gamma$ as the inverse of the square root of the distance and has a jump across $\Gamma$. Moreover $g_\om =(u_{down}-u_{up})\cdot \overrightarrow{\t}$. This result was extended in \cite{GV-Lac} for any $\om_0 \in L^q_c(\R^2)$, $q\in (2,\infty]$, and without assuming any regularity for the curve.

Moreover, for any solution of the above system in the case of $C^2$ curve and $\om_0 \in L^\infty_c(\R^2)$, it was established in \cite{lac-uni} that the solution is a renormalized solution in the sense of DiPerna-Lions, hence we have the following extra properties
\begin{itemize}
\item the $L^p$ norm of the vorticity is conserved, for any $p\in [1,\infty]$;
\item the circulation around $\Gamma$ is conserved ;
\item the vorticity is always compactly supported, but this support can grow ;
\item around the curve, the velocity at times $t>0$ belongs to $L^{p}_{\loc}(\R^2)\cap W^{1,r}_{\loc}(\R^2)$ only for $p<4$ and $r<4/3$. Actually we have
\[u\in L^{\infty}(\R_+; L^p(B(0,R))) \cap L^{\infty}([0,T]; L^s(\R^2)) \cap L^{\infty}(\R_+;W^{1,r}(B(0,R)))\]
for any $R>0$, $T>0$, $p\in [1,4)$, $s\in (2,4)$ and $r\in [1,4/3)$.
\end{itemize}
If we assume that $\om_0$ belongs to $L^\infty_c(\R^2\setminus \Gamma)$ and $(\om_0,\g)$ has a definite sign (namely, either $\om_0$ non-negative and $\g_0<-\int \om_0$ or $\om_0$ non-positive and $\g_0> -\int \om_0$) then the uniqueness of the global weak solution is the main result of \cite{lac-uni}. The crucial step therein is to prove that the sign condition implies that the vorticity never meets the boundary, which is the place where the velocity is not regular.

\bigskip

The natural question is to wonder if an infinite number of material points has the same effect as a material curve on the behavior of an inviscid flow. Let us consider $k$ disjoints balls uniformly distributed on the ``imaginary'' curve $\Gamma$, and we look for the limit as $k\to \infty$. To well constructed the initial data, we choose $\{\g_i^k\}$ such that $\sum_{i=1}^k \g_i^k z_i^k$ tends to the measure $g_{\om_0}(s)\delta_{\Gamma}$ when $k\to \infty$ (typically, we set $\g_i^k$ the sum of $g_{\om_0}$ around $z_i^k$). Therefore, our main theorem states that there exists $\varepsilon_k\to 0$ such that the Euler solution on $\R^2 \setminus (\cup_{i=1 \dots k} B(z_i^k,\varepsilon_k))$ converges to a pair $(u,\om)$ verifying, in the sense of distributions, the following system
\begin{equation*}
\left\lbrace\begin{aligned}
&\pd_t \om+u\cdot \na \om=0 &\text{ in } \R^2\times (0,\infty) \\
&\div u =0 &\text{ in } \R^2\times (0,\infty) \\
&\curl u =\om +g_{\om_0}(s) \d_\Gamma &\text{ in } \R^2\times (0,\infty) \\
&\om(x,0)=\om_0(x) &\text{ in } \R^2.
\end{aligned}\right.
\end{equation*}

The difference of the two systems lies in the third equation, where the density $g$ depends on $\om(t,\cdot)$ for the exterior of the curve whereas it is independent of time in the last system. Even if this difference seems tiny, it gives an important consequence on the behavior of the flow, that we comment now. In the exterior of a curve, the presence of the additional measure $g_{\om}(s) \d_\Gamma$ is mandatory to get the tangency condition, because $K_{\R^2}[\om]$ has no reason to be tangent. The only conserved quantity in this case is the global circulation, i.e. the sum of $g_{\om}$. But for $x\in \Gamma$, $g_{\om} (t,x)$ can be not constant, in order to counterbalance the normal part of $K_{\R^2}[\om(t,\cdot)](x)$. In the case of an infinite number of point, we have decompose the vortex sheet $g_{\om_0} \d_\Gamma$ by $k$ point vortices, but there is a principal consequence of this cut. In the case of an ideal flow, the circulation of the velocity around an obstacle is conserved, which means that the densities $\g_i^k$ of the points vortices are constant. Therefore, for any time, $\sum_i \g_i^k z_i^k$ approximate $g_{\om_0} \d_\Gamma$ instead of $g_{\om}(t,\cdot) \d_\Gamma$.

To conclude, we have carefully chosen $\g_i^k$ such that the limit velocity is initially tangent to the curve, but there is no reason that it remains tangent: $u=K_{\R_2}[\om(t,\cdot) + g_{\om_0}(s) \d_\Gamma]$ in this case whereas $u=K_{\R_2}[\om(t,\cdot) + g_{\om(t,\cdot)}(s) \d_\Gamma]$ in the case of the exterior of the curve. With the lack of the tangency condition, it appears possible that the vorticity $\om$ meets the curve in finite time, which would make it difficult to get the global uniqueness (even with a sign condition).

\subsection{Rate size of the disks/space between the disks}\label{sect:rate}\

We comment in this section the fact that $\varepsilon_k$ is constructed in the main theorem, and is not initially given. In our analysis, we have constructed a family $\varepsilon_k$ such that the convergence holds true, but it could mean that the size the obstacles tends more rapidly to $0$ than the number of balls $k$ tends to $\infty$. This can explain why, in the limit of a continuous curve, the effective flow may cross the curve.

An interesting question is to study the limit when $k\to \infty$ for a family $\varepsilon_k$ given. Such a question is the purpose in \cite{BLM,LM}: let us consider $k$ obstacles regularly distributed on the unit segment. We assume that the size of the obstacles is $\varepsilon_k$ and the minimal distance is $d_k$ ($\leq 1/k$). For $\om_0\in L^\infty_c(\R^2)$ given, we consider the Euler solution $u^k$ with initial vorticity $\om_0\vert_{\OM^k}$ and zero initial circulations. Theorem~\ref{main} states that there exists $\varepsilon_k\to 0$ such that $u^k$ convergences to the Euler solution in the full plane (here $\mu=0$), i.e. we do not feel at the limit the infinite number of material points. However, we do not control the rate $\varepsilon_k$ over $d_k$ (linked to $n_{k}$ and $\varepsilon_k$), which can be very small.

Conversely, let us fix $(\varepsilon_k)$ and $(d_{k})$ two sequences tending to zero and we study the limit $k\to \infty$. The main result in \cite{BLM} reads as follows: in this particular setting (zero initial circulations and uniformly distributed on a segment), then the limit $u^k$ converges to the Euler solution in the full plane if $d_{k}=\varepsilon_{k}^\alpha$ with $\alpha<2$. If we distribute the obstacles on the unit square, then $u^k$ converges to the Euler solution in the full plane if $\alpha<1$. The main idea in that paper is to look for the critical $\alpha$ such that the Step 1 in Subsection~\ref{sect : C} could give the convergence to zero in $L^2$ of $K_{\R^2}[f\vert_{\Omega^{\e}}]-Ev^\e[f]$. Next, the authors have remarked that the term $v^\e[f]-K^\e[f\vert_{\Omega^{\e}}]-\sum_{i=1}^n m_i^\e[f] X_i^\e$ is the Leray projection of $-(K_{\R^2}[f\vert_{\Omega^{\e}}]-Ev^\e[f])$, and they concluded by the orthogonality of this projection in $L^2$. This argument allows them to skip the analysis on the harmonic part (Step 2 in Section~\ref{sect : B}), which is the hardest part in this article. However, we aware that such an argument cannot work without assuming that all the circulations vanish. As recalled previously, if we assume that $\gamma_1^k$ is constant, then $\gamma_1\delta_{z_1}$ appears at the limit, which implies that the velocity behaves like $ \gamma_1\frac{(x-z_1)^\perp}{2\pi |x-z_1|^2}$ close to $z_1$, which does not belong in $L^2$. Therefore, $L^2$ estimates are not easy in our case, and it is not clear that the Leray projection is uniformly continuous in $L^p$ for $p<2$ (see a discussion in \cite[Section 5.3]{LM}). It explains why we need the Step 2 of Section~\ref{sect : B} ($L^p$ framework) in our case.

In \cite{LM}, the authors have completed the result of \cite{BLM}: they exhibit some regimes where the limit of $u^k$ is a solution to the Euler equations in the exterior of the unit segment or the unit square. That result relies again on an $L^2$ argument, and holds only when all the initial circulations are zero. When the obstacles are distributed on the segment, the impermeable wall appears if the distance $d_{k}$ is much smaller than the size $\varepsilon_{k}$, namely such that $\frac{d_{k}}{\varepsilon_{k}^3}\to 0$ if the obstacle is a disk (see \cite{LM} for more details).

\bigskip

\noindent
 {\bf Acknowledgments.}
 C.L. is partially supported by the Agence Nationale de la Recherche, Project DYFICOLTI grant ANR-13-BS01-0003-01, and Project IFSMACS grant ANR-15.
H.J.N.L.'s research was supported in part by CNPq grant \# 307918/2014-9 and FAPERJ project \# E-26/103.197/2012. M.C.L.F.'s work was partially funded by CNPq grant \# 306886/2014-6.
 This work has been supported by the R\'eseau Franco-Br\'esilien en Math\'ematiques and by the project PICS-CNRS ``FLAME'', both of which made several visits between the authors possible.


\def\cprime{$'$}

\end{document}